\documentclass{article}

	\usepackage{amsmath}
	\usepackage{amssymb}
	\usepackage{amsthm}

	\usepackage{tikz}
	\usepackage{graphicx}

	\usepackage{authblk}
	
	\usepackage{accents}

	\newcommand{\nI}{\noindent}

	\newcommand{\RR}{\mathbb{R}}
	\newcommand{\NN}{\mathbb{N}}
	\newcommand{\ZZ}{\mathbb{Z}}
	\newcommand{\CC}{\mathbb{C}}
	
	\newcommand{\XX}{\mathbb X}	

	\newcommand{\alf}{\alpha}
	\newcommand{\bet}{\beta}
	\newcommand{\ga}{\gamma}
	\newcommand{\de}{\delta}
	\newcommand{\eps}{\varepsilon}
	\newcommand{\la}{\lambda}
	\newcommand{\ka}{\kappa}
	\newcommand{\om}{\omega}
	\newcommand{\sig}{\sigma}

	\newcommand{\De}{\Delta}
	\newcommand{\Ga}{\Gamma}

	\newcommand{\cali}{\mathcal I}
	\newcommand{\calo}{\mathcal O}
	\newcommand{\calp}{\mathcal P}

	\newcommand{\id}{\mathrm{d}}

		\newcommand{\boldj}{{\boldsymbol{j}}}
		\newcommand{\boldw}{{\boldsymbol{w}}}
		\newcommand{\boldv}{{\boldsymbol{v}}}

		\newlength{\dhatheight}
		
		\newcommand{\doublehat}{%
			\settoheight{\dhatheight}{\ensuremath{\hat{t}}}%
			\addtolength{\dhatheight}{-0.3ex}%
			\hat{\vphantom{\rule{1pt}{\dhatheight}}%
				\smash{\hat{t}}}}
				
		\newcommand{\subdoublehat}{%
			\settoheight{\dhatheight}{\ensuremath{\hat{t}}}%
			\addtolength{\dhatheight}{-0.8ex}%
			\hat{\vphantom{\rule{1pt}{\dhatheight}}%
				\smash{\hat{t}}}}
	
		\newlength{\thatheight}
		
		\newcommand{\triplehat}{%
			\settoheight{\thatheight}{\ensuremath{\doublehat}}%
			\addtolength{\thatheight}{-0.3ex}%
			\hat{\vphantom{\rule{1pt}{\thatheight}}%
				\smash{\doublehat}}}

		\newcommand{\emp}{\emptyset}

		\DeclareMathOperator{\It}{It}

		\DeclareMathOperator{\leb}{Leb}

		\newcommand{\one}[1]{\raisebox{-2pt}{\scalebox{1.3}{$\chi$}\vphantom{\big(}}_{\{#1\}}}
		
		\newcommand{\laz}{\la\to 0^+}
	
		\newcommand{\volcic}{Vol\v ci\v c}

	\theoremstyle{plain}
		\newtheorem{prop}{Proposition}
		\newtheorem{thm}{Theorem}
		\newtheorem{thmo}{Theorem}
		\newtheorem*{thm*}{Theorem}
		\newtheorem*{cor}{Corollary}
		\newtheorem{lem}{Lemma}

	\theoremstyle{definition}
		\newtheorem*{defn}{Definition} 
		\newtheorem{exam}{Example}
		\newtheorem*{rem}{Remark}

	\title{An infinite interval version  of the $\alf$-Kakutani equidistribution problem }
	\author{M. Pollicott\thanks{Partly supported by ERC-Advanced Grant 833802-Resonances and EPSRC grant
EP/T001674/1} {} and B. Sewell}
\affil{University of Warwick}
\begin{document}
	
	\maketitle
\begin{abstract}
In this article  we extend results of Kakutani, Adler--Flatto, Smilansky and others 
	on the classical $\alpha$-Kakutani equidistribution result for sequences arising from finite partitions of the interval.  In particular, we describe a generalization of the equidistribution result to infinite partitions.   In addition, we give discrepancy estimates, extending results of Drmota--Infusino \cite{drmota}.
	\end{abstract}

\section{Introduction}
	
	Uniform distribution of sequences of numbers $(x_n)_{n=1}^\infty$ in the unit interval has been an important area of interest for over a century.  For example, it was shown by Weyl \cite{weyl} that, for any irrational $\alpha$, the sequence $x_n = \alpha n$ (mod $1$) is uniformly distributed and Hardy and Littlewood showed that, for almost all $\lambda > 1$, the sequence $x_n = \lambda^n$ (mod $1$) is uniformly distributed \cite{hardy}.		
	In this note we consider another simple family of examples based on subdividing intervals.
	Before introducing the original motivating example, we first fix 
	our terminology: a \textit{partition} $\calp$ is a set of closed, positive-length intervals, which have pairwise disjoint interiors and cover $[0,1]$ up to a set of Lebesgue measure zero.

	\subsection{The original $\alf$-Kakutani equidistribution result}
		
	In 1973
	Araki posed a problem which lead to Kakutani to prove an elegant equidistribution result. (An interesting historical background is presented in \cite{adler-flatto}).
For clarity, we give a description of his original partition scheme, which we generalise in the next section.
	
	\begin{defn}[\(\alf\)-Kakutani scheme]
		For a fixed  \(0 < \alf < 1\), the \textit{\(\alf\)-Kakutani scheme} is a sequence of partitions \((\calp _n)_{n=0}^\infty\) defined inductively:
		\begin{itemize}
			\item  \(\calp_0 = \big\{[0,1]\big\}\) is the trivial partition; and 
			\item \(\calp _{n + 1}\) is obtained from  \(\calp _{n}\)  by taking each  interval of maximal length  and subdividing it into two smaller intervals  in the ratio \(\alf:1-\alf\).
		\end{itemize} 
	\end{defn}
	\begin{exam}
	 	Figure \ref{fig - 1/3 Kakutani} shows the first seven partitions for the choice $\alf = 1/3$. Notice that $\calp_5$ is obtained by splitting two maximal length intervals in $\calp_4$ simultaneously (each of length 2/9). By contrast, the choice $\alf = 1/2$ gives the trivial dyadic splitting.
	
		\begin{figure}[h]
			$$\begin{tikzpicture}[scale=6]
			
			\draw[||-||] (0, 0) node[anchor=east]{$\calp_0$} -- (1.00, 0);
			\draw[|-||] (1, -0.143) -- (0.333, -0.143);
			\draw[|-] (0, -0.143) node[anchor=east]{$\calp_1$} -- (0.333, -0.143);
			\draw[|-|] (0,-0.286) node[anchor=east]{$\calp_2$} -- (0.333, -0.286);
			\draw[||-|] (0.556, -0.286) -- (0, -0.286) -- (1.00, -0.286);
			\draw[|-|] (0, -0.429) node[anchor=east]{$\calp_3$} -- (0.333, -0.429);
			\draw[|-||] (0.556, -0.429) -- (0.704, -0.429);
			\draw[|-] (1.00, -0.429) -- (0, -0.429);
			\draw[|-||] (0, -0.571) node[anchor=east]{$\calp_4$} -- (0.111, -0.571);
			\draw[|-|] (0.333, -0.571) -- (0.556, -0.571);
			\draw[|-|] (0.704, -0.571) -- (0, -0.571) -- (1.00, -0.571);
			\draw[|-|] (0, -0.714) node[anchor=east]{$\calp_5$} -- (0.111, -0.714);
			\draw[||-|] (0.185, -0.714) -- (0.333, -0.714);
			\draw[||-|] (0.407, -0.714) -- (0.556, -0.714);
			\draw[|-|] (0.704, -0.714) -- (0, -0.714) -- (1.00, -0.714);
			\draw[|-|] (0, -0.857) node[anchor=east]{$\calp_6$} -- (0.111, -0.857);
			\draw[|-|] (0.185, -0.857) -- (0.333, -0.857);
			\draw[|-|] (0.407, -0.857) -- (0.556, -0.857);
			\draw[|-||] (0.704, -0.857) -- (0.802, -0.857);
			\draw[|-] (1.00, -0.857) -- (0, -0.857);
			\draw[|-|] (0, -1.00) node[anchor=east]{$\calp_7$} -- (0.111, -1.00);
			\draw[|-|] (0.185, -1.00) -- (0.333, -1.00);
			\draw[|-|] (0.407, -1.00) -- (0.556, -1.00);
			\draw[|-|] (0.704, -1.00) -- (0.802, -1.00);
			\draw[||-|] (0.868, -1.00) -- (0, -1.00) -- (1.00, -1.00);
			
			\path (0,0) node [above = 3pt] {$0$};
			\path (1,0) node [above = 3pt] {$1$};
			\path (1/3,-1/7) node [above = 3pt] {$\frac 13$};
			\path (5/9,-2/7) node [above = 3pt] {$\frac 59$};
			\path (19/27,-3/7) node [above = 3pt] {$\frac {19}{27}$};
			\path (1/9,-4/7) node [above = 3pt] {$\frac 19$};
			\path (11/27,-5/7) node [above = 3pt] {$\frac {11}{27}$};
			\path (5/27,-5/7) node [above = 3pt] {$\frac {5}{27}$};
			\path (65/81,-6/7) node [above = 3pt] {$\frac {65}{81}$};
			\path (211/243,-1) node [above = 3pt] {$\frac {211}{243}$};

		\end{tikzpicture}
		$$
		\caption{The first seven partitions $(\calp_n)_{n=0}^7$ of the $\frac13$-Kakutani scheme.}
		\label{fig - 1/3 Kakutani}
	\end{figure}
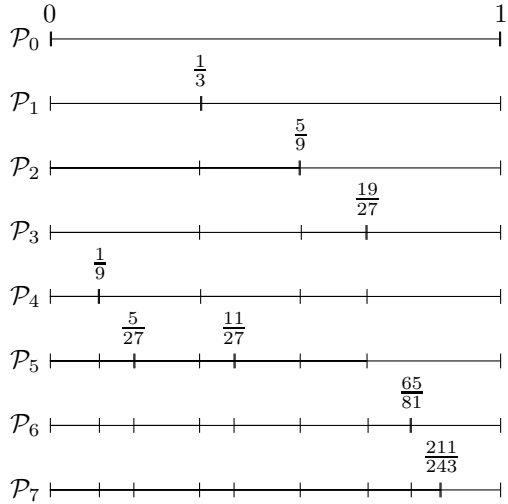
		\end{exam}
		
Consider the set of endpoints at the $n$th stage of this process, \(E_n = \bigcup_{I \in \calp_n} \partial I\). Kakutani's result is the following. 
	%
	\begin{thm*}[Kakutani]
		For all $\alf \in (0,1)$, the set \(E_n\) becomes uniformly distributed as \(n \to \infty\).
	\end{thm*}
\subsection{Interval substitutions using multiple intervals}
	
	A natural generalisation of the $\alf$-Kakutani scheme, first introduced by \volcic{} in \cite{volcic 11}, is to alter the above process by splitting intervals of maximal length
	according to a fixed finite partition
	 consisting of $N \geq 2$ subintervals, say. 
	That is, at each stage, one splits all intervals of maximal length 
	into $N$ pieces whose lengths have a certain fixed ratio, $\alf_1:\alf_2:\dots:\alf_N$, where the $\alf_i$ sum to 1. 
	%
	
	\begin{exam}
		In Figure~\ref{fig - 1/3 12 1/6 picture}, we have the first seven partitions of the interval substitution scheme in which one splits maximal intervals according to the 
		partition $\big\{[0,\frac12],[\frac12,\frac23],[\frac23,1]\big\}$, i.e., with 
		ratio $\frac{1}{2} : \frac{1}{6} : \frac{1}{3}$.
		By contrast, the $\alf$-Kakutani scheme corresponds to splitting according to the partition $\big\{[0,\alf],[\alf,1]\big\}$.
		\label{exam - 1/2 1/3 1/6}
		\begin{figure}[h]
			$$
			\begin{tikzpicture}[scale=6]
				
				\draw [||-||] (0, 0) -- (1, 0);
				\draw [|-|] (0, -1/7) -- (1, -1/7);
				\draw [|-|] (0, -2/7) -- (1, -2/7);
				\draw [|-|] (0,-3/7) -- (1, -3/7);
				\draw [|-|] (0, -4/7) -- (1, -4/7);
				\draw [|-|] (0, -5/7) -- (1, -5/7);
				\draw [|-|] (0, -6/7) -- (1, -6/7);
				\draw [|-|] (0, -1) -- (1, -1);
				
				
				\foreach \x in {{1/2},{2/3}}
				\draw[|-||] (0,-1/7) -- (\x,-1/7);
				
				\path (0,0) node [above = 3pt] {$0$};
				\path (1,0) node [above = 3pt] {$1$};
				\path (1/2,-1/7) node [above = 3pt] {$\frac 12$};
				\path (2/3,-1/7) node [above = 3pt] {$\frac 23$};
				\path (1/4,-2/7) node [above = 3pt] {$\frac 14$};
				\path (1/3,-2/7) node [above = 3pt] {$\frac 13$};
				\path (5/6,-3/7) node [above = 3pt] {$\frac 56$};
				\path (8/9,-3/7) node [above = 3pt] {$\frac 89$};
				\path (0.142,-4/7) node [above = 5pt] {$\ldots$};			
				
				\foreach \x in {{1/4},{1/3},{1/2},{2/3}}
				\draw[-|] (0,-2/7) -- (\x,-2/7);
				
				\foreach \x in {{1/4},{1/3},{1/2},{2/3},{5/6},{8/9}}
				\draw[-|] (0,-3/7) -- (\x,-3/7);
				
				\foreach \x in{{1/8},{1/6},{1/4},{1/3},{1/2},{2/3},{5/6},{8/9}}
				\draw[-|] (0,-4/7) -- (\x,-4/7);
				
				\foreach \x in {{1/8},{1/6},{1/4},{1/3},{5/12},{4/9},{1/2},{7/12},{11/18},{2/3},{3/4},{7/9},{5/6},{8/9}}
				\draw[-|] (0,-5/7) -- (\x,-5/7);
				
				\foreach \x in {{1/16},{1/12},{1/8},{1/6},{1/4},{1/3},{5/12},{4/9},{1/2},{7/12},{11/18},{2/3},{3/4},{7/9},{5/6},{8/9}}
				\draw[-|] (0,-6/7) -- (\x,-6/7);
				
				\foreach \x in {{1/16},{1/12},{1/8},{1/6},{1/4},{1/3},{5/12},{4/9},{1/2},{7/12},{11/18},{2/3},{3/4},{7/9},{5/6},{8/9},{17/18},{26/27}}
				\draw[-|] (0,-1) -- (\x,-1);
				
				
				\foreach \x in {{1/4},{1/3}}
				\draw[|-||] (0,-2/7) -- (\x,-2/7);
				\foreach \x in {{5/6},{8/9}}
				\draw[|-||] (0,-3/7) -- (\x,-3/7);
				\foreach \x in {{1/8},{1/6}}
				\draw[|-||] (0,-4/7) -- (\x,-4/7);
				\foreach \x in {{5/12},{4/9},{7/12},{11/18},{3/4},{7/9}}
				\draw[|-||] (0,-5/7) -- (\x,-5/7);
				\foreach \x in {{1/16},{1/12}}
				\draw[|-||] (0,-6/7) -- (\x,-6/7);
				\foreach \x in {{17/18},{26/27}}
				\draw[|-||] (0,-1) -- (\x,-1);

				\foreach \x in {0,1,...,7}
				\path (0, -\x/7) node[anchor=east]{$\mathcal {P}_\x$};
				
			\end{tikzpicture}
			$$
					\caption{The first seven partitions $(\calp_n)_{n=0}^7$ of the interval substitution scheme where one splits maximal-length intervals according to the partition $\calp_1 = \big\{[0,\frac12],[\frac12,\frac23],[\frac23,1]\big\}$.}
			\label{fig - 1/3 12 1/6 picture}
		\end{figure}
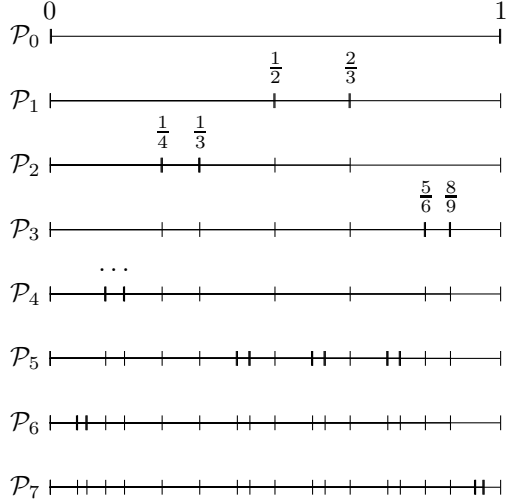
	\end{exam}
	
	One particular family of these schemes, that of so-called LS-sequences, which include the $\phi$-Kakutani scheme for $\phi = \frac 12 (\sqrt 5 -1)$, has received particular attention in the context of low discrepancy sequences: see, e.g., \cite{ais ais,carbone,carbone I V,Hi Tichy,weiss}.
	
	\subsection{Interval substitutions using infinitely many intervals}
	In this note, we continue the process and ask in what way the result above still holds if at every stage we insert
	an \textit{infinite} partition $\calp$
	into each maximal-length subinterval.
	We  denote by  $\hat E_n$  the finite set of endpoints of those intervals which have been split up to the $(n+1)$-st stage, i.e., 
		$$
			\hat E_n := \left\{ \min(I),\; \max(I) \;:\; I \in \bigcup _{i=0}^n \calp_i 	\,\setminus \,\calp_{n+1} \right\}.
		$$
	\begin{exam}
		In Figure \ref{fig - 1/2 1/3 to n example} we depict  $\calp_n$ and $\hat E_n$ (for $n \leq 7$) for the infinite substitution scheme generated by $\calp = \big\{[0,\frac12]\big\}\cup\big\{ [1 - \frac12\cdot 3^{-n}, 1 - \frac16\cdot3^{-n}]\big\}_{n=0}^\infty$.

\begin{exam}
	A wilder example: let the intervals in $\calp$ be the connected components of the compliment of the middle third Cantor set.
\end{exam}
		
		\begin{figure}
			$$\begin{tikzpicture}[scale=6]			
						
					\foreach \x in {0,1,...,7}{
						\draw[|-|] (0, -\x/7) node[left=4pt]{$\calp_\x$} -- (1,-\x/7);
						\draw[dotted] (1.1,-\x/7 + 1/28) node[right=2pt] {$\hat E_\x$} -- (0,-\x/7 + 1/28);
					}
				
				
				\foreach \x in {0.5,0.833,0.944,0.981,0.994,0.998,0.999}
				\draw[|-|] (0, -1/7) -- (\x, -1/7);
				\foreach \x in {0.25,0.417,0.472,0.491,0.497,0.499,0.5,0.833,0.944,0.981,0.994,0.998,0.999}
				\draw[|-|] (0, -2/7) -- (\x, -2/7);
				\foreach \x in {0.25,0.417,0.472,0.491,0.497,0.499,0.5,0.667,0.778,0.815,0.827,0.831,0.833,0.944,0.981,0.994,0.998,0.999} \draw[|-|] (0, -3/7) -- (\x, -3/7);
				\foreach \x in {0.125,0.208,0.236,0.245,0.248,0.249,0.25,0.417,0.472,0.491,0.497,0.499,0.5,0.667,0.778,0.815,0.827,0.831,0.833,0.944,0.981,0.994,0.998,0.999} 
				\draw[|-|] (0, -4/7) -- (\x, -4/7);
				\foreach \x in {0.125,0.208,0.236,0.245,0.248,0.249,0.25,0.333,0.389,0.407,0.414,0.416,0.417,0.472,0.491,0.497,0.499,0.5,0.583,0.639,0.657,0.664,0.666,0.667,0.778,0.815,0.827,0.831,0.833,0.944,0.981,0.994,0.998,0.999}
				\draw[|-|] (0, -5/7) -- (\x, -5/7);
				\foreach \x in {0.062,0.104,0.118,0.123,0.124,0.125,0.208,0.236,0.245,0.248,0.249,0.25,0.333,0.389,0.407,0.414,0.416,0.417,0.472,0.491,0.497,0.499,0.5,0.583,0.639,0.657,0.664,0.666,0.667,0.778,0.815,0.827,0.831,0.833,0.944,0.981,0.994,0.998,0.999}
				\draw[|-|] (0, -6/7) -- (\x, -6/7);
				\foreach \x in {0.062,0.104,0.118,0.123,0.124,0.125,0.208,0.236,0.245,0.248,0.249,0.25,0.333,0.389,0.407,0.414,0.416,0.417,0.472,0.491,0.497,0.499,0.5,0.583,0.639,0.657,0.664,0.666,0.667,0.722,0.759,0.772,0.776,0.777,0.778,0.815,0.827,0.831,0.833,0.889,0.926,0.938,0.942,0.944,0.981,0.994,0.998,0.999}
				\draw[|-|] (0, -1) -- (\x, -1);
				
				
				\foreach \x in {{(0,1/28)},{(1,1/28)},{(0,-3/28)},{(1/2,-3/28)},{(1,-3/28)},{(0,-1/4)},{(1/2,-1/4)},{(5/6,-1/4)},{(1,-1/4)},{(0,-11/28)},{(1/4,-11/28)},{(1/2,-11/28)},{(5/6,-11/28)},{(1,-11/28)},{(0,-15/28)},{(1/4,-15/28)},{(5/12,-15/28)},{(1/2,-15/28)},{(2/3,-15/28)},{(5/6,-15/28)},{(1,-15/28)},{(0,-19/28)},{(1/8,-19/28)},{(1/4,-19/28)},{(5/12,-19/28)},{(1/2,-19/28)},{(2/3,-19/28)},{(5/6,-19/28)},{(1,-19/28)},{(0,-23/28)},{(1/8,-23/28)},{(1/4,-23/28)},{(5/12,-23/28)},{(1/2,-23/28)},{(2/3,-23/28)},{(7/9,-23/28)},{(5/6,-23/28)},{(17/18,-23/28)},{(1,-23/28)},{(0,-27/28)},{(1/8,-27/28)},{(5/24,-27/28)},{(1/4,-27/28)},{(1/3,-27/28)},{(5/12,-27/28)},{(1/2,-27/28)},{(7/12,-27/28)},{(2/3,-27/28)},{(7/9,-27/28)},{(5/6,-27/28)},{(17/18,-27/28)},{(1,-27/28)}}
				{
					\draw[fill=yellow] \x circle [radius = 0.01];
				}
			\end{tikzpicture}
			$$
			
			\caption{An illustration of $(\calp_n)_{n=0}^7$ and $(\hat E_n)_{n=0}^7$ for the example generated by the partition $\calp = \big\{[1,\frac12]\big\}\cup\big\{ [1 - \frac12\cdot 3^{-n}, 1 - \frac16\cdot3^{-n}]\big\}_{n=0}^\infty$. Here the tick marks (which accumulate on certain points in the interval) denote the elements of $E_n$ and the suspended yellow circles denote the elements of $\hat E_n$.}
			
			\label{fig - 1/2 1/3 to n example}
		\end{figure}
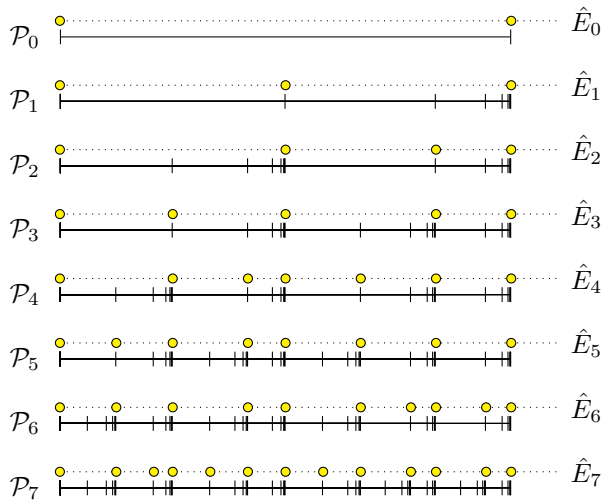
	\end{exam}
	For simplicity, we restrict our attention to the set of left endpoints, which we shall denote by $L_n$, although we could equally well have chosen the right endpoints, midpoints, etc.
	
	Our main result is the following generalization of Kakutani's
	equidistribution theorem.   Let $\|I\|$ denote the length of an interval $I$.

	\begin{thmo}
	Let $\calp$ be a countable partition. Then, provided that
			$$
				-\displaystyle \sum_{I \in \calp} \|I\|\log \|I\| < \infty,
			$$ 
		%
		the set $L_n$ becomes uniformly distributed as $n \to \infty$.
		\label{thmk-basic equidistribution}
	\end{thmo}

		A brief outline of this note: In section 2, we give a new dynamical viewpoint of the problem and in section 3, we apply renewal theory to prove Theorem \ref{thmk-basic equidistribution} in two cases.
	In section 4 we use a generating function to estimate the discrepancy in the rank one case.
	In section 5 we use methods of analytic number theory to estimate the discrepancy in the higher rank case, with a generic Diophantine-type assumption.
	
Our interest in this problem, and the starting point for our analysis, began with the  
very elegant work of 
Smilansky \cite{smila}.	
	
\section{Partitions and similarities}
Our approach to Theorem \ref{thmk-basic equidistribution} is to express the elements of the partition $\calp$ in terms of the images of similarities.  The refinements into finer partitions, $\calp_n$, can then be expressed in terms of words formed from the index set of $\calp$.

That is, one can write $\calp = \{T_i[0,1]\}_{i \in \cali}$, where each $T_i:[0,1] \to [0,1]$ is an orientation preserving similarity with contraction ratio $\alf_i > 0$. We see that $\calp$ being a partition is equivalent to the following:
	\begin{itemize}
		\item $T_i[0,1)\cap T_j [0,1) = \emp$ for $i \neq j$; and 
		\item $\displaystyle \sum_{i \in \cali} \alf_i = 1$.
	\end{itemize}
	\begin{exam}
		Given $0 = t_0 < t_1 < \cdots$ with $(t_n) \to 1$, the partition $\{[t_{n-1},t_n]\}_{n \in \NN}$ is equal to $\{T_n[0,1]\}_{n\in\NN}$, where 
			$$
				T_n(x) = (t_n - t_{n-1})x + \sum _{k=0}^{n-1}	t_n.
			$$
	\end{exam}
	\begin{exam}
		Setting $t_n = 1 - \frac 16\, 3^{-n}$ for $n\geq 1$ gives rise to $\calp_1$ in Figure \ref{fig - 1/2 1/3 to n example}. 
	\end{exam}
	We now explain how this can be used to give an explicit description of the splitting process.
	\begin{defn}[\((T_i)\)-refinement] Given a partition $\calp = \{S_k[0,1]\}_k$, where the \(\{S_k\}_k\) are orientation preserving similarities, the  \((T_i)\)-refinement of \(\calp\) is the refinement obtained by taking all intervals of maximal length in \(\calp\) and replacing them by subintervals in the following manner: if $S[0,1] \in \calp$ has maximal length in $\calp$, then it is replaced by the elements of the set
		$$
			\{S\circ\,T_i[0,1]\;|\;i\in \cali\}.
		$$
	\end{defn}
	\begin{defn}[Interval substitution scheme]
		The interval substitution scheme generated by \(\{T_i\}_{i \in \cali}\) is the sequence of partitions \((\calp_n)_{n=0}^\infty\) defined as follows:
		\begin{itemize}
			\item \(\calp_0\) is the trivial partition, \(\calp_0 = \big\{[0,1]\big\}\); and 
			\item \(\calp_{n + 1}\) is the \((T_i)\)-refinement of \(\calp_n\).
		\end{itemize} 
	\end{defn}
	This gives a convenient presentation of the partitions.
	
	\begin{exam}
		The \(\alf\)-Kakutani scheme is the interval substitution scheme generated by the pair \ \(T_1:x \mapsto \alf x\) \ and \ \(T_2:x \mapsto  (1 - \alf) x + \alf\). 
	\end{exam}
	\begin{exam}
		Similarly, the interval substitution scheme generated by the triple $T_1:x \mapsto x/2$, $T_2: x \mapsto  (x+3)/6$ and $T_3:x \mapsto (x+2)/3$ gives the sequence of partitions depicted in Figure \ref{fig - 1/3 12 1/6 picture}.
	\end{exam}	

	We now associate to the sequence of partitions $(\calp_n)_{n=0}^\infty$ a sequence of families of left endpoints of split intervals, $(L_n)_{n=0}^\infty$.
	
	\begin{defn}[$L_n$]
		Given an interval substitution scheme $(\calp_n)_{n=0}^\infty$ generated by similarities  $(T_i)_{i \in \cali}$, we define the finite sets $L_n$  ($n \geq 0$) to be
			$$
					L_n 
				= 
					\bigcup_{k=0}^n 
						\ %
						\bigcup_{I \in \calp_k\setminus\calp_{k+1}}
							\min(I).
			$$
	\end{defn}
	\begin{rem}
		One can consider a generalisation of the above process by dropping the assumption that the $(T_i)_i$ have to be affine. We will not consider this more general setting, but we note it is easy to give superficial examples where $E_n$ (or $L_n$) is not uniformly distributed: take, for example, $T_1(x) =\sqrt x/2$, $T_2(x) = (x + 1)/2$.
	\end{rem}
	Considering the interval substitution scheme generated by $\{T_i\}_{i \in \cali}$, it follows inductively that every interval appearing in the process is obtained by applying a sequence of maps from $\{T_i\}_i$ to $[0,1]$, and so each is naturally described by a finite word in $\cali$. It is convenient to introduce the following notation. 
		\begin{defn}[$W(\cali)$, $\ast$, $\alf_\boldv$, $T_\boldv$]
			Given a countable set $\cali$, the word set $W(\cali)$ is the semigroup consisting of all words in $\cali$: i.e.
				$$
					W(\cali) = \{\emp\} \cup \bigcup_{n = 1}^\infty \cali^n,
				$$
			where $\emp$ denotes the empty word (unique word of length zero), and the semigroup operation $\ast:W(\cali)\times W(\cali) \to W(\cali)$ denotes concatenation of words, for which $\emp$ acts as the identity:
				$$
					(n_1,\ldots,n_k)\ast(m_1,\ldots,m_j) = (n_1,\ldots,n_k, m_1,\ldots,m_j);
						\qquad 
					\boldv \ast \emp = \emp\ast \boldv = \boldv.
				$$
			Furthermore, for ease of notation, we extend the definitions of $\alf_i$ and $T_i$ to the whole of $W(\cali)$: For the word $\boldv = (i_1,\ldots, i_k) \in \cali^k$, define
					$$
						\alf_\boldv := \prod_{j=1}^k\alf_{i_j},\qquad T_\boldv := T_{i_1} \circ\ldots\circ T_{i_k},
					$$
			and also define
			 $\alf_\emp = 1$ and $T_\emp = \text{Id}_{[0,1]}$.
		\end{defn}
	%
	That is, given any interval $I$ which appears in the process, $I = T_\boldv[0,1]$ for some word $\boldv \in W(\cali)$\footnote{Similarly, every word $\boldv \in \cali$ gives rise to an interval in some $\calp_n$}, and will be split between $\calp_n$ and $\calp_{n+1}$ (i.e., $I \in \calp_n \setminus \calp_{n-1}$) precisely when $n$ satisfies
		$$
			\alf_\boldv = \max_{I \in \calp_n}\big\{\|I\|\big\},
		$$
	and consequently its left endpoint $T_\boldv(0)$ will appear in $L_n$, if not already present in $L_{n-1}$.
	
	Rather than using $n$ in $\{L_n\}_{n \geq 1}$ to parametrise this process, we want to reparameterise this family to reflect the lengths of the maximal intervals, and rewrite it as $(X_\la)$ as follows.
	\begin{defn}
			For $\la>1$, let $X_\la = \emp$, and for $\la \in (0,1]$, let
				$$
						X_\la 
					:= 
						L_{n(\la)},
					\quad
						\text{where}
					\quad
						n(\la) := \max\{n \ge 0 \,:\, \forall I \in \calp_n,\ \|I\|\geq \la\}.
				$$
	\end{defn}
	\nI I.e.,  given $\la\leq 1$, $P_{n(\la)+1}$ is the first partition in the process for which all intervals have lengths strictly smaller than $\la$. From the previous discussion, one obtains a convenient, dynamical formula for $X_\la$:
			$$
					X_\la
				= 
					\big\{
						T_\boldv(0)
					\,:\,
						\boldv\in W(\cali),
						\,
						\alf_\boldv \geq \la
					\big\}.
			$$
	As $\laz$, $n(\la) \to \infty$ and so the uniform distribution of $L_n$ as $n\to \infty$ is  equivalent to that of $X_\la$ as $\laz$.
	As an aside, this is equivalent to the convergence of the probablity measure $\mu_\la$,
		$$
			\mu_\la = \frac1{|X_\la|}\sum _{x \in X_\la} \delta_x,
		$$
	 to the Lebesgue measure (which we will henceforth denote as $\leb$), as $\laz$, where $\delta_x$ denotes the Dirac delta measure at $x$, and $|\cdot|$ denotes cardinality (of a finite set).

	\section{Proof of Theorem 1}
	This section is devoted to proving the main result, which we can conveniently rephrase in the following way.
		\begin{thm}
			Provided that
			$-\sum_{I\in \calp} \|I\|\log \|I\| < \infty$,
			the measures $\mu_\la$ converge to the Lebesgue measure $\leb$ as Borel measures as $\laz$, i.e., for any interval $J \subset [0,1]$, we have $\mu_\lambda(J) \to \leb(J) = \|J\|$. 

		\end{thm}
	We first prove the convergence for a given interval of the form $T_\boldv[0,1)$. For each of these elementary sets, their $\mu_\la$-measure is intimately related to the asymptotics of $|X_\la|$ as $\laz$.
	To proceed, we  can relate  $|X_\la|$ to the set of words $|A_\la|$, where
		$$
			A_\la := \{\boldw \in W(\cali)\,|\,\alf_\boldw \geq \la \},
		$$
	and exploit  a natural renewal equation for the quantity $\la |A_\la|$. 
		%
	\begin{lem}
		One of the following two cases hold. Either there is a (unique) symbol in $1 \in \cali$,  say, such that $T_1(0) = 0$; in which case, for all $\la > 0$,
			\begin{equation}
				|A_\la| = \sum_{k=0}^\infty |X_{\la\alf_1^{-k}}|\ \iff\ |X_\la| = |A_\la|- |A_{\la/\alf_1}|;
			\label{equation - Al & Xl in fixed case}
			\end{equation}
		or no element of $\{T_\boldw\}_{\boldw \in W(\cali)}$ fixes 0, and $|X_\la| = |A_\la|$ for all $\la > 0$.
	\end{lem}
	\begin{proof}[Proof of Lemma 1]
		Let $\boldw, \boldv \in A_\la$ satisfy $T_\boldw(0) = T_\boldv(0)$. If $\boldw \neq \boldv$, the disjointness of $T_i[0,1)$ and the injectivity of the $T_i$ imply inductively that one of these words is obtained from the other by concatenation with a word fixing the identity; i.e., without loss of generality, $\boldw = \boldv \ast \boldj$, where $\boldj \in W(\cali)$ satisfies $T_\boldj(0) = 0$. We now have the two cases:
		
		If there is no symbol $i \in \cali$ for which $T_i(0) = 0$, we must have $\boldj = \emp$; which implies the map $\boldv \mapsto T_\boldv(0)$ is a bijection $A_\la \to X_\la$.
			
		For the other case, let $1 \in \cali$ be such that $T_1(0) = 0$; it is unique by the disjointness of $T_i[0,1)$. Then $\boldj \neq \emp$ only if $\boldj \in \{1\}^k \subset \cali^k$ for some $k \in \NN$, i.e., $\boldj$ is a tuple of 1's. It follows that, for each $y \in X_\la$, there is a unique word $\boldv_0(y) \in A_\la$ satisfying
		\begin{itemize}
			\item $T_{\boldv_0(y)}(0) = y$; and
			\item $T_\boldv(0) = y \implies \boldv = \boldv_0(y) \ast \boldj$, for some $\boldj \in \{\emp\}\cup\{1\}^k $, $k \in \NN$.
		\end{itemize}
		In particular, $T_\boldv(0) = y$ implies $\alf_\boldv = \alf_{\boldv_0(y)} \alf_1^k$, for some $k\in\NN_0$. In other words, there is exactly one element of $A_\la \setminus A_{\la/\alf_1}$, $n_0(y)$, which gets mapped onto $y$.	Therefore, $\boldv \mapsto T_\boldv(0)$ gives a bijection $A_\la \setminus A_{\la/\alf_1} \to X_\la$. The right hand side of \eqref{equation - Al & Xl in fixed case} follows, completing the proof.
	\end{proof}

	Now, to continue the proof of the theorem, we combine the previous lemma with the following, which expresses  $\mu_\la(T_\boldv[0,1))$ as a ratio involving $|X_\la|$.
	
	\begin{lem}
		For all $\boldv \in W(\cali)$ and $\la \in (0,1]$,
			\begin{equation}
				|T_\boldv[0,1) \cap X_{\la \alf_\boldv}| = |X_\la|.
					\label{eq - Tn X alfla}
			\end{equation}
		In particular,  for all $\la\in (0,\alf_\boldv]$,
			\begin{equation}
				\mu_\la(T_\boldv[0,1)) = \frac {|X_{\la/\alf_\boldv}|}{|X_\la|}.
					\label{eq - Tn X alfla COR}
			\end{equation}
		Moreover, if $\boldv = \emp$, or $\boldv = \boldw \ast i$ with $T_i(0) \neq 0$, these hold for all $\la>0$.
	\end{lem}

	\begin{proof}[Proof of Lemma 2]
		 Fix $\boldv \in W(\cali)$ and consider \eqref{eq - Tn X alfla}. First let $\la \leq 1$. By induction, using the properties of the $T_i$, if $T_\boldw(0) \in T_\boldv[0,1)$, then, either $\boldw = \boldv \ast \boldj$ or $\boldv = \boldw \ast \boldj$, for some $\boldj \in W(\cali)$. Moreover, if $\boldw \in A_{\la\alf_\boldv}$, the second option gives a contradiction: $\alf_\boldw > \alf_\boldv \geq \la \alf_\boldv \geq \alf_\boldw$. Thus, we have
			$$
				T_\boldv[0,1) \cap X_{\la\alf_\boldv} = \{T_{\boldv\ast \boldj}(x) \,:\,\boldj \in A_\la\}.
			$$
		By injectivity of $T_\boldv$, the right hand side bijectively corresponds to $X_\la$, giving \eqref{eq - Tn X alfla}. Applying \eqref{eq - Tn X alfla}, with $\la/\alf_\boldv$ in place of $\la$ and dividing by $|X_\la|$ yields \eqref{eq - Tn X alfla COR}.
	
		As for the final remark, now let $\la>1$. This implies the second option holds, $\boldv = \boldw \ast \boldj$, since the first is satisfied only if $\la \leq 1$, by a similar contradiction argument. It follows that $T_\boldj(0) = 0$ and thus, as in the proof of Lemma 1, $\boldj \in \{1\}^k$ for some $k \in \NN$. This contradicts the assumptions of the remark; thus both the left and the right hand side of $\eqref{eq - Tn X alfla}$ are empty when $\la>1$. The lemma follows.
	\end{proof}

	The significance of relating $|X_\la|$ to $|A_\la|$ will now become clear from   the following renewal equation.
	\begin{lem}
		The following holds for all $\la > 0$.
			\begin{equation}
				|A_{\la}| = \sum_{i \in \cali}|A_{\la/\alf_i}| + \one{\la \leq 1},
					\label{equation - renewal equation for A}
			\end{equation}
		where $\chi$ is the indicator function.
		Equivalently, the following renewal equation applies, for $Z(t) := e^{-t} |A_{e^{-t}}|$.
			$$
					Z(t) 
				= 
					\sum_{i \in \cali}
						\alf_i 
						Z\big(t - \log({\alf_i}^{-1})\big) + e^{-t} \one{t \ge 0},
				\qquad
					  \forall\, t \in \RR.
			$$
	\end{lem}
	\begin{proof}[Proof of Lemma 3]
		The second equation is a restatement of the first, \eqref{equation - renewal equation for A}. One can obtain \eqref{equation - renewal equation for A} by partitioning the non-empty words in $A_\la$ according to their first symbol. That is, one can write the following disjoint union.
			$$
					A_\la\setminus \{\emp\}
				= 
					\bigsqcup_{i \in \cali} \underbrace{\{i\ast \boldv \in W(\cali)\,|\,\alf_\boldv\leq \la/\alf_i\}}_{\text{in bijection with }A_{\la/\alf_i}}.
			$$
		This gives rise to the sum in \eqref{equation - renewal equation for A}. Accounting for $\emp$, $\emp \in A_\la$ if and only if $\la \leq 1$, thus providing the indicator term in \eqref{equation - renewal equation for A}, and completing the proof.
	\end{proof}

	To make use of this renewal equation, just as in \cite{aistleitner-starting partitions, drmota, smila} it is necessary to consider two cases which behave somewhat differently. These cases correspond to, for example, the $\alf$-Kakutani schemes for $\alf = 1/3$ and $\alf = 1/2$, as described in the introduction.
	\begin{defn}[Rank] 
		For $n \in \NN$ we will say the collection $\{\alf_i\}_{i \in \cali}$ is \textit{rank $n$} if the smallest additive subgroup of $\RR$ containing the set $\{-\log(\alf_i)\}_{i \in \cali}$ is isomorphic to $\ZZ^n$. If $\{\alf_i\}_i$ is not rank $n$ for any $n\in \NN$ we will say $\{\alf_i\}_{i \in \cali}$ is \textit{infinite rank}. Also, whenever $\{\alf_i\}_i$ is not rank one, we say it is \textit{higher rank}.
	\end{defn}

	\begin{exam}
	 The following examples illustrate the different ranks:
		\begin{enumerate}
			\item $\{1/2^n\}_{n \in \NN}$ is rank one;
			\item $\{1/2\} \cup \{1/3^n\}_{n \in \NN}$ is rank two;
			\item $\{1/2\} \cup \{1/3\} \cup \{1/7^n\}_{n \in \NN}$ is rank three; and 
			\item $\{1/n^s\}_{n = 2}^\infty$ is infinite rank, where $s\approx 1.728$ satisfies $\zeta(s) = 2$ (here $\zeta$ denotes the Riemann zeta function).
		\end{enumerate}
	\end{exam}

	\subsection{Uniform distribution in the rank one case}
	In  this subsection, we concentrate on the rank one case, also known as the \textit{arithmetic}, \textit{rationally-related}  or \textit{commensurable} case (see \cite{aistleitner-starting partitions,drmota,smila}). The characteristic feature of this case is that the contraction ratios are all powers of a common number, $\{\alf_i\}_{i \in \cali} \subset \{x^n\}_{n\in \NN}$. Thereby, fixing the minimal such $x>0$, equation \eqref{equation - renewal equation for A} defines a discrete renewal equation on the lattice $-\log(x)\ZZ$, and one can apply the Erd\H os-Feller-Pollard renewal theorem (see \cite{erdos-feller-pollard}) to obtain the following lemma.
	\begin{lem}[Erd\H os-Feller-Pollard renewal theorem]
		Suppose that $\{\alf_i\}_{i \in \cali}$ is rank one, that $x>0$ is the minimal positive number for which $\{\alf_i\}_{i \in \cali} \subset \{x^n\}_{n\in \NN}$, and that $H := -\sum_i\alf_i\log(\alf_i) < \infty$, then
			$$
				Z\big(-n\log(x)\big) \to \frac  1 H \quad \text{as }n\to \infty,\ n \in \NN;
			$$
			where $Z( \cdot )$ is defined in Lemma 3.
	\end{lem}
	This lemma yields the following corollary, all but completing the proof of the theorem in this case.
	\begin{cor}
		For $\{\alf_i\}_{i \in \cali}$ and $H$ as in the previous lemma, the following hold as $n \to \infty$, $n \in \NN$:
			\begin{itemize}
				\item $|A_{x^n}|\sim x^{-n}/H$.
				\item For all $\boldv \in W(\cali)$, $\mu_{x^n}(T_\boldv[0,1)) \to \alf_\boldv = \leb(T_\boldv[0,1))$.
			\end{itemize}
	\end{cor}
	\begin{proof}[Proof of Corollary]
		The first item is simply a restatement of the conclusion of Lemma 4. The second item follows with an application of Lemmas 1 and 2. For $n \in \NN$ we have one of the following, according to the two cases of Lemma 1:  Either no $T_i$ fixes $0$, and
			$$
					\mu_{x^n} (T_\boldv[0,1))
				= 
					\frac
						{|A_{x^n\alf_\boldv^{-1}}|}
						{|A_{x^n}|}
				\sim
					\frac
						{\big(\alf_\boldv^{-1} x^n\big)^{-1}/H}
						{x^{-n}/H}
				= 
					\alf_\boldv;
			$$
			or, for $T_1(0) = 0$ and $\boldv_0$ the longest subword of $\boldv$ not ending in a 1, 
			$$
				\begin{aligned}
						\mu_{x^n} (T_\boldv[0,1))
					&= 
						\frac
							{
								 |A_{{x^n}\alf_\boldv^{-1}}|
								-|A_{{x^n}(\alf_\boldv\alf_1)^{-1}}|
								+ \one{\alf_\boldv < {x^n}\leq \alf_{\boldv_0}}
							}
							{|A_{{x^n}}| - |A_{{x^n}\alf_1^{-1}}|}
						\cr
					&\sim 
						\frac
							{1-\alf_1}
							{1-\alf_1}
						\frac
							{\alf_\boldv x^{-n}/H}
							{x^{-n}/H}
					= 
						\alf_\boldv \cr
				\end{aligned}
			$$
		as $n \to \infty$, as required.
	\end{proof}
	%

		%
		%

The last, easy, step in the proof of Theorem 1 concerns packing intervals of the form $T_\boldv[0,1)$ into a given interval. This will be written as if in the continuous case. For the rank one case, for $\la\to 0^+$, one can read $\la = x^n,$ $n \to \infty$; in fact, it is a trivial matter to prove they are equivalent.
\begin{proof}[Proof of Theorem 1]
	\label{page - Thm 1 proof}
	Let $I \subset [0,1]$ be an interval, and let $\alf_{\max} = \max_{i \in \cali}(\alf_i)$. For $n \in \NN$, let
		$$
			U_n := \{\boldv \in \cali^n \,:\, T_\boldv[0,1) \subset I\}.
		$$
	We claim the total length of intervals corresponding to $U_n$ approximates the length of $I$, as follows:
		\begin{equation}
				\sum_{\boldv \in U_n} 
					\|T_\boldv[0,1)\|
			\geq
				\|I\| - 2 \alf_{\max}^n.
			\label{eq - packing intervals into I}
		\end{equation}
	To prove this, take $x \in I \setminus \bigcup_{\boldv \in U_n} T_\boldv[0,1)$. Then one of the following hold:
		\begin{enumerate}
			\item $x \in K_n := [0,1] \setminus \bigcup_{\boldv \in \cali^n} T_\boldv[0,1)$. An inductive argument shows, for all $\la$ and $n$, $\mu_\la(K_n) = 0 = \text{Leb}(K_n)$.
			\item $x \in T_\boldw[0,1)$ for some $\boldw \in \cali^n\setminus U_n$. Since $T_\boldw[0,1) \not \subset I$, it is an interval meeting $\partial I$; thus there are at most two $\boldw \in \cali^n$ with this property.
		\end{enumerate}
	In other words, $I \setminus \bigcup_{\boldv \in U_n} T_\boldv[0,1) \setminus K_n$ comprises at most two intervals, each with length at most $\alf_{\max{}}^n$. \eqref{eq - packing intervals into I} follows.
	
	To apply this, let $U$ be a union of a finitely number of intervals in $U_n$, and let $\|U\|$ denote its total length (i.e., $\leb(U)$). Then $\mu_\la(I) \geq \mu_\lambda(U)$ and
		$$
				\liminf_{\la \to 0^+}
					\mu_\lambda(I)
			\geq 
				\lim_{\la \to 0^+}
					\mu_\la(U)
			=
				\|U\|.
		$$
	Taking the supremum over all such finite unions $U$ gives
		$$
				\liminf_{\la \to 0^+}
					\mu_\lambda(I)
			\geq 
				\|S\|
			\geq
				\|I\| - 2\alf_{\max}^n.
		$$

	Repeating the above argument for the 0, 1 or 2 intervals comprising $[0,1]\setminus I$ yields a converse inequality for $\|I\|$:
		$$
				\limsup_{\la \to 0^+}
					\mu_\lambda(I)
			\leq 
				\|I\| + 2\alf_{\max}^n,			
		$$
	and the proof is completed by taking $n \to \infty$.
\end{proof}
%



	%
\subsection{Uniform distribution in the higher rank case}
		In the remaining, generic case, the proof is very similar to the above. It continues with the following lemma, a convenient application of the Blackwell renewal theorem (see \cite{blackwell}).
		
	\begin{lem}[Blackwell renewal theorem]
		Suppose that $\{\alf_i\}_{i \in \cali}$ is not rank one and that $H = -\sum_i\alf_i\log(\alf_i) < \infty$. Then one has the continuous limit
			$$
				Z(t) \to \frac 1H 
				\quad \text{as }t\to \infty,
			$$
		where $Z(\cdot)$ is defined in Lemma 3.
		\label{lem - blackwell renewal}
	\end{lem}
This gives the following corollary. The proof is similar to that of the preceding, with the convergence now taking place as $\laz$.
	\begin{cor}
		For $\{\alf_i\}_{i \in \cali}$ as in the previous lemma, the following hold as $\laz$.
		\begin{itemize}
			\item $|A_{\la}|\sim 1\big/\la H$,
			\item For all $\boldv \in W(\cali)$, $\mu_{\la}(T_\boldv[0,1)) \to \alf_\boldv = \leb(T_\boldv[0,1))$.
		\end{itemize}
	\end{cor}
	The conclusion of Theorem 1 in this case is the same as given in the proceeding case, see p.\pageref{page - Thm 1 proof}.
	\begin{rem}
	The same method of proof gives a more general result.
	Suppose we have a set $\XX$ equipped with a probability measure $\nu$, and there is a collection $\{T_i,\alf_i\}_{i \in \cali}$ such that
	\begin{itemize}
		\item[a)] $T_i:\XX \to \XX$ are injective $\nu$-measurable functions;
		
		\item[b)] for all $i,j \in \cali$ distinct, $T_i(\XX) \cap T_j(\XX) = \emp$;
		
		\item[c)] for all $\boldv \in W(\cali)$, $\nu (T_\boldv(\XX)) = \alf_\boldv$;
		
		\item[d)] $\alf_i > 0$ for all $i \in \cali$ and $\sum_{i \in \cali} \alf_i = 1$;
		{} and
		
		\item[e)] $-\sum_{i \in \cali}\alf_i \log(\alf_i) < \infty$.
	\end{itemize}
	Then, for any $x \in \XX\setminus \bigcup_{i \in \cali} T_i(\XX)$, and for any $\nu$-measurable set $S$ which can be written as
		$$
			S = \bigcup_{\boldv\in V} T_\boldv(\XX)
		$$
	(with $V$ is any subset of $w(\cali)$),
	we have, as $\laz$,
		$$
			\frac{|S \,\cap\, X_\la(x)|}{|X_\la(x)|} \to \nu(S),
		$$
	where $X_\la(x) = \{T_\boldv(x)\,:\; \boldv \in W(\cali),\ \alf_\boldv\geq \la\}$.
	
	It is also possible to generalise the method of proof to include substitution schemes starting from arbitrary partitions, and one recovers results analogous to those of \cite{aistleitner-starting partitions}.
	
	These will be discussed in greater depth in the forthcoming doctoral thesis of the second author.
\end{rem}

\section{Discrepancy estimates}
In \cite{volcic 11}, the author both generalised  the method used by Adler and Flatto in \cite{adler-flatto} to general finite partitions, and posed questions which inspired various other papers. Of particular interest is the behaviour of the \textit{discrepancy}.  This corresponds to estimating the speed of convergence (where the $I$ are intervals) of
	$$
		 	\sup_{I\subset[0,1]}\big|\, \mu_{\la}(I)  - \|I\|\,\big|
		 \to 
		 	0 
		 \hbox{ as }
		 	\la \to 0^+.
	$$	

The Koksma inequality uses such estimates to give convergence rates for integrals of bounded-variation functions. Hence, fast-decaying discrepancies provide potential computationally-efficient numerical integration techniques. The theory of discrepancies (of equidistributing sets and sequences) has naturally received a lot of attention (see, e.g., \cite{drmota tichy} for an overview). There are also interesting open problems, for example on optimal discrepancy in higher dimensions.

Regarding interval substitution schemes, in the finite-partition case, discrepancy estimates are provided by Drmota--Infusino in \cite{drmota}, extending the application-focused work of Carbone in \cite{carbone}. We in turn extend this to the context of infinite partitions. The results obtained here are different, depending on whether we are in the rank one case or the higher rank case.
	
\subsection{Discrepancy estimates in the rank one case}
In this final section, we extend the analysis of the rank one case to estimate the discrepancy between the measure $\mu_\la$ and the Lebesgue measure. More precisely, we have the following result.
	\begin{thm}
		Suppose that 
		\begin{enumerate}
			\item $\{\alf_i\}_{i \in \cali}$ is rank one; 
			\item $x>0$ is the smallest number for which $\{\alf_i\}_{i \in \cali} \subset \{x^n\}_{n\in \NN}$; and
			\item there is some $\eps > 0$ for which $\sum_{i \in \cali}\alf_i^{1-\eps} < \infty$.
		\end{enumerate}
		Then there is an $ R^\ast \in (0,1)$, made explicit in Lemma 6 below, such that, for all $\rho \in (x/R^\ast,1)$, there is a constant $C>0$ such that, for all $n \in \mathbb N$ and all intervals $I \subset [0,1]$,
			$$
				\big| \,\mu_{x^n}(I)  - \|I\|\,\big| \leq C\rho^n.
			$$		
	\end{thm}	
The proof of Theorem 2 begins with the following light lemma, defining $R^\ast$ in terms of a generating function for $|A_{x^n}|$.

	\begin{lem}
		Given $\{\alf_i\}_{i \in \cali}$ as in Theorem 2, the function formally defined by
			$$
				g(z) = (z-x)\sum_{n = 0}^\infty |A_{x^n}|z^n
			$$
		has a holomorphic extension to the open disk of radius $R^\ast$ about 0, where
			$$
					R^\ast
				:= \min \left(
					\{x^{1-\eps}\}
						\cup
					\big\{|z|\;:\: z \in \CC\setminus \{x\},\ \sum_{j \in \cali}z^{n_j} =1 \big\}
					\right) 
				> x
			$$
		and where $n_i := \log_x(\alf_i)$. Therefore, denoting by $b_n$ the $n$th Taylor coefficient of $g$, given $R < R^\ast$, there exists $C>0$ such that, for all $n\in \NN$, $b_n\leq C R^{-n}$.
	\end{lem}	
	\begin{proof}[Proof of Lemma 6]
		From \eqref{equation - renewal equation for A}, the renewal equation of Lemma 3, one has, for $|z|\leq x^{1-\eps}$ and $z \neq x$,
			\begin{align*}
				\frac{g(z)}{z-x} = \sum_{n=0}^\infty  |A_{x^n}| z^n
					& =
						\sum_{n=0}^\infty 
							\left(\sum_{j \in \cali} |A_{x^{n - n_j}}| + 1 \right)
							z^n
					\\
					& =
						\sum_{n=0}^\infty 
							\sum_{j \in \cali} 
								|A_{x^{n - n_i}}| z^n 
					+
						\frac 1 {1-z}
					\\
					& =
						\sum_{j \in \cali} 
							z^{n_j}
						\sum_{n = 0}^\infty  
							|A_{x^{n-n_j}}| z^{n-n_i}   
					+  
						\frac 1 {1 - z}
					\\
					& =
						\sum_{j \in \cali} 
							z^{n_j} 
							\frac{g(z)}{z-x}  
					+ 
						\frac 1 {1 - z},
			\end{align*}
		which rearranges to
			$$
				g(z) = \frac{z-x}{(z-1)(\sum_{j \in \cali}z^{n_j} - 1)}.
			$$
		Therefore, $g$ has a meromorphic expansion on the disk of convergence of
			\begin{equation}
					z 
				\mapsto 
					\sum_{j \in \cali}
						z^{n_j};
				\label{eq - power sum with n_i's}
			\end{equation}
		which has radius at least $x^{1-\eps}$, by assumption on the decay of $\{\alf_i\}$.
		
		An elementary argument (see \cite[pp.201--2] {erdos-feller-pollard}) shows that $z = x$ is not only a simple root of \eqref{eq - power sum with n_i's} with residue $1/H$ (see the statement of Lemmas 4 or 5), but it is also the only root of \eqref{eq - power sum with n_i's} in the closed disk $\{|z| \leq x\}$.
%
%
		 Therefore, %
		 $g$ is holomorphic on the open disk of radius $R^\ast > x$, where $R^\ast$ is the absolute value of the next smallest root of \eqref{eq - power sum with n_i's}, or equal $x^{1-\eps}$ if no other root exists.
	\end{proof}
	\begin{exam}
		In certain nice cases, one can say more. For the simplest infinite example, $\alf_n = 2^{-n}$ ($n \in \NN$), $g(z) \equiv 1/2$ is constant.
	\end{exam}
The final stage of the proofs of Theorems 2 and 3 is similar to that of Theorem 1, but the method of splitting up a general interval needs a little more care. We will only prove the case that none of the $T_i$ fix $0$, since the other case is similar but tedious.
	\begin{proof}[Proof of Theorem 2]
		For simplicity, consider the interval $I = [b,1)$, for fixed $b \in (0,1)$, and assume that no $i \in \cali$ satisfies $T_i(0) = 0$.
		For $n \in \NN$, let $V_n$ denote the elements of $U_n$ (where $U_n$ is in the proof of Theorem 1) whose interval is not contained in one from $U_{n'}$ for any $n' < n$. More explicitly,
				$$
						V_1
					:=	
						\big\{i \in \cali \;:\; T_i[0,1) \subset I\big\}
					=
						U_1,
				$$
			and, for $n \geq 2$,
				\begin{align*}
						V_{n}
					&:=
						\big\{
							\boldv \ast i \in \cali^n
						\;:\; 
							i \in \cali,\ %
							T_{\boldv \ast i}[0,1) \subset I\text{ but }%
							T_{\boldv}[0,1) \not \subset I
						\big\}
					\\
					&\phantom{:}=
						U_n \setminus U_{n-1}.
				\end{align*}

			It is simple to show the union over all intervals coming from the $V_n$,
				$$
						\bigcup_{n = 1}^\infty
							\bigcup_{\boldv \in V_n}
									T_\boldv[0,1),		
				$$
			is disjoint, contained in $I$, and differs from $I$ by at most some exceptional set of points---those contained in at most finitely many $T_\boldv[0,1)$, i.e., a subset of
				$$
						K
					=
						[0,1] 
						\setminus
						\bigcup_{N\in\NN}
							\bigcap_{n>N}
								\bigcup_{\boldv \in \cali^n}
									T_\boldv[0,1).
				$$
			This $K$, similarly to $K_n$ in the proof of Theorem 1, has both $\mu_\la$ and Lebesgue measure zero.
				

			To say more, it is necessary to give a partial description for $V_n$, involving the itinerary for $b$.
			 
			 \begin{defn}[Itinerary, $\It(x)$, $\It_n(x)$]
			 	This definition has two cases. First suppose that $x \in [0,1]$ is such that
				 	\begin{itemize}
				 		\item $x$ is a left endpoint, $T_\boldv(0)$, for some $\boldv\in W(\cali)$; or
				 		\item $x$ lies in the exceptional set $K$ above.
				 	\end{itemize} 
			 	Then
			 		$$
			 				n 
			 			=
			 				\min \left\{
			 					k \in \NN
			 				: 
			 					x \notin \bigcup_{\boldv \in \cali^k}T_\boldv[0,1)
			 				\right\}
			 		$$	
		 		exists; and we call $\It(x)$, the \textit{itinerary} of $x$, the unique word in $\cali^{n-1}$ such that
		 			$$
		 				x \in T_{\It(x)}[0,1).
		 			$$
		 		and we say $x$ has \textit{finite itinerary}.
		 		
		 		Otherwise, we say that $x$ has \textit{infinite itinerary}, and the itinerary $\It(x)$ is the sequence 
		 			$$
		 					\It(x) 
		 				:= 
		 					(i_n)_n \in \cali^\NN$$
		 		such that, for each truncation $\It_n(x) := (i_1,\ldots,i_n) \in \cali^n$, $x \in T_{\It_n(x)}[0,1)$.
	 	 	\end{defn}
	 	 	
	 	 	Returning to the proof: From a similar argument to that in Theorem 1, if $\It(b) = (i_1,i_2,\ldots,i_n) \in \cali^n$ has finite itinerary, then $V_k$ is empty for all $k \geq n+2$, and also for all $k \leq n$, we have
				\begin{align}
						I_k 
					&= 
						\{(i_1,i_2,\ldots, i_{k-1}, i) \in \cali^k \ |\  T_{i_k}(0) < T_{i}(0) \}
							\label{eq - I_k for itinerary}
					\\
					&\subset \{\It_k(b)\} \times \cali \qquad \text{and similarly,}
							\nonumber
					\\
						I_{n+1}
					&\subset
						\{\It(b)\} \times \cali.
							\nonumber
				\end{align}	 	 	
	 	 		Otherwise, if $\It(b) = (i_n)_{n=1}^{\infty}$ is infinite, \eqref{eq - I_k for itinerary} holds for all $k \in \NN$ (it is as if $n = \infty$).

Now let $V = \bigcup_{k \in \NN} V_k$ and $n \in \NN$. It follows from the nullity of $K$ that we can write the following, and we divide the sum into two, corresponding to intervals which have or haven't been split at this $n$th stage:
			\begin{align}
					\mu_{x^n}(I) - \|I\|
				& =  
					\sum_{\boldv \in V}
						\mu_{x^n}(T_\boldv[0,1)) - \alf_\boldv
					\nonumber
				\\
				& =
					\sum
							_{\substack{\boldv \in V \\ \alf_{\boldv} \geq x^n}}
						\frac
						{|X_{x^n/\alf_\boldv}|- \alf_{\boldv}|X_{x^n}|}
						{|X_{x^n}|}
				\quad-
					\sum_{\substack{\boldv \in V \\ \alf_{\boldv} < x^n}} \alf_\boldv.
					\label{eq-two sums for discrepancy}
			\end{align}
		We estimate the second sum first, corresponding to intervals which have not yet been split up to this value of $n$. Firstly, we have
			$$	
					0
				\leq
					\sum_{\substack{\boldv \in V \\ \alf_{\boldv} < x^n}} 
						\alf_\boldv 
				\leq
					x^{n(1-\eps)}
					\sum_{\substack{\boldv \in V \\ \alf_{\boldv} < x^n}} 
						\alf_\boldv^{1-\eps},
			$$
		and the sum on the right hand side can be bounded uniformly in $b$, as follows. Recall that, in the infinite itinerary case, the inclusion in equation \eqref{eq - I_k for itinerary} holds for all $n\in \NN$, and we may write
			\begin{alignat*}{4}
					\sum_{\boldv \in V} 
						\alf_\boldv^{1-\eps}
				=
					\sum_{k = 1}^\infty
						\sum_{\boldv \in V_k} 
					 		\alf_\boldv^{1-\eps}
				& \leq	
					\sum_{k = 1}^\infty
				 		\alf_{\It_{n-1}(x)}^{1-\eps}
				 	&&	
				 	\sum_{i \in \cali} \alf_i^{1 - \eps}
				\\ & \leq
					\sum_{n=1}^\infty 
						(\alf_\ast^n)^{1-\eps}
					&&	
					 \sum_{i \in \cali} 
					 	\alf_i^{1 - \eps}		 		
				\qquad
						\big(\alf_\ast = \max_{i \in \cali}\{\alf_i\}\big)
				\\ & = 
				 		\frac 
				 			{(\alf_\ast)^{1-\eps}}
				 			{1 - (\alf_\ast)^{1-\eps}} 
				 		&&\sum_{i \in \cali} \alf_i^{1 - \eps}
				 \qquad =: C
			\end{alignat*}
		\label{page-uncounted intervals decay}
		which is finite by assumption. The finite itinerary case is similar and we obtain the same bound, $C$: the only difference is that sum in $k$ is finite.
		
		Therefore, the second sum of \eqref{eq-two sums for discrepancy} is bounded above by $C x^{n(1-\eps)}$, and we may turn our attention to the first.

		Consider the term of this first sum corresponding to $\boldv \in V$. Write $\alf_\boldv = x^m$ for some $m \in \NN$, and consider a generating series for (the numerator of) the corresponding summand, which relates to the $g$ from Lemma 6:
			\begin{align*}
					\sum_{n = 1}^\infty \big(|X_{x^{n-m}}| - x^m|X_{x^n}|\big)z^n
				& = 
					(z^m - x^m)
					\left( 
						\frac{g(z)}{z-x}
					\right)
				\\& =
				(z^{m-1} + x z^{m-2} + \cdots + x^{m-1}) g(z).
			\end{align*}
		Recalling $b_n$ as the $n$th Taylor coefficient of $g$, equating coefficients on both sides gives, for all $m \geq n$,
			$$
					|X_{x^{n-m}}| - x^m|X_{x^n}| 
				=
					b_{n-m+1} + x b_{n-m+2} + \cdots + x^{m-1} b_n.
			$$
		Thus, by Lemma 6, given $\rho \in (x/R^\ast,1)$, there is some (possibly updated) constant $C>0$ such that, for all $n\in \NN$, $b_n < C(\rho/x)^n$. Applying this to the previous equation gives, for all $n \geq m$,
			$$
					|X_{x^{n-m}}| - x^m|X_{x^n}|
				\leq 
					C \frac {\rho^{n-m-1} + \rho^{n-m-2} + \cdots + \rho^n} {x^{n-m+1}}
				= C \left(\frac \rho x\right)^n \frac{(x/\rho)^m - x^m}{x(1-\rho)},
			$$
		and dividing both sides through by $|X_{x^n}| = \calo(x^{-n})$ gives
			$$
					\left|
						\frac{|X_{x^n/\alf_\boldv}| - \alf_{\boldv}|X_{x^n}|}{|X_{x^n}|}
					\right|
				\leq 
					\hat C \rho^n \left(
						\left(
							\frac x\rho
						\right)^{\!\!m}
						-
						x^m
					\right)
				<
					\hat C \rho^n
					\left(
						\alf_\boldv^{1-\eps} - \alf_\boldv
					\right),
			$$
		where we have used $x/\rho < (R^\ast)^{-1} \leq x^{1-\eps}$ in the last inequality, for some constant $\hat C = \hat C(x,\rho)>0$. Summing over $\boldv \in V$ bounds the first sum of \eqref{eq-two sums for discrepancy}:
		$$
				\sum_{\substack{\boldv \in V \\ \alf_{\boldv} \geq x^n}}
				\left|
					\frac{|X_{x^n/\alf_\boldv}| - \alf_{\boldv}|X_{x^n}|}{|X_{x^n}|}
				\right|
			\leq
				\left(
					\hat C\sum_{\substack{\boldv \in V\\ \alf_{\boldv}\geq x^n}}
						\alf_\boldv^{1-\eps} - \alf_\boldv
				\right)
				\rho^n
			\leq
				\hat C (C + 1) \rho^n.
		$$
	Since $n \in \NN$ was arbitrary, we have the required estimate.
	\end{proof}
\subsection{Discrepancy estimates in the higher rank case}
When the collection $\{\alf_j\}_{j \in \cali}$ is not rank one, we require not only a strict decay property on the $\{\alf_j\}$---as in Theorem 2---, but also a kind of Diophantine condition. To proceed, we need the following definition.
	\begin{defn}[$R$-badly approximable]
		For $R \in [2,\infty)$ we say a number $\ga \in \RR$ is $R$-badly approximable if there exists a $d>0$ such that
			$$
				\forall\, (l,k) \in \ZZ^2\ s.t.\ l \neq 0,\quad \left| 
				\ga
				 - \frac k l \right| > \frac d {|l|^{R}}.
			$$
	\end{defn}
	\begin{rem} Larger values of $R$ correspond to more easily approximable numbers:
		\begin{itemize}
			\item For any $R=2$, the property is equivalent to $\ga$ having bounded continued fraction coefficients. (Such $\ga$ comprise a set of measure 0 containing all quadratic algebraic numbers.)
			\item  For $R>2$, the property holds Lebesgue almost-everywhere: by Jarnik's theorem \cite[Thm. 10.3]{falconer}, the Hausdorff dimension of the complementary set is $2/R < 1$.
			\item Certain transcendental numbers (e.g., \textit{Liouville numbers}) do not satisfy the property for any $R$ whatsoever.
		\end{itemize}
	\end{rem}
	\begin{thm}
		Suppose that $\{\alf_i\}_{i \in \cali}$ is not rank one, that there is some $\eps>0$ such that $\sum_i \alf_i^{1-\eps}<\infty$, and that there is a pair $\alf_j, \alf_k \in \{\alf_i\}_{i \in \cali}$ such that $\log(\alf_j)/\log(\alf_k)$ is $(2 + r)$-badly approximable, for some $r \in [0,1/2)$. Then, for all $P \in (0,P^*)$, there exists a constant $C$ such that, for all intervals $I \subset [0,1]$,
			$$
				\left| \, \mu_\la(I) - \|I\| \,\right| \leq  C\big(-\log(\la)\big)^{-P};
			$$
		where
			$$
				P^* = \frac{1-2r}{8(1+r)}.
			$$
	\end{thm}
The proof of Theorem 3 requires us to consider the Mellin transform,
	$$
		g(z) = \int_0^\infty t^{-z-1} |A_{1/t}|\ \id t,
	$$
which has the following explicit form, courtesy of the renewal equation for $|A_\la|$. Let $\Re$, $\Im$ denote the real and imaginary parts of a complex number, respectively.
\begin{lem}
	For $\Re(z) > 1$, the Mellin transform $g(z)$ takes the form
		$$
			g(z) = \frac 1{z(\sum_{j \in \cali}{\alf_j}^{\!z} - 1)}.
		$$
	In particular, if $\ \sum_j {\alf_j}^{\!1-\eps}< \infty$, $g$ has a meromorphic extension to the half-plane $\{\Re(z) > 1 - \eps\}$.
	\label{lem - Mellin transform made explicit}
\end{lem}
\begin{proof}[Proof of Lemma 7]
	The formula follows from standard properties of the Mellin transform, namely $g(\alf_j z) = \alf_j^z g(z)$, and the renewal equation of Lemma 3.
\end{proof}
As is well-known, one can obtain asymptotic information about a function from the distribution of poles of its Mellin transform.

In particular, this depends on the zeros of the almost-periodic function
	$$
		f(z) = \sum_{j \in \cali}{\alf_j}^{\!z} - 1.
	$$
There is a lot we can say straight away. From our assumptions on $\{\alf_j\}_j$, we have that $z=1$ is a simple zero of $f$, and by the triangle inequality, there are no zeros of $f$ for $\{\Re(z) > 1 \}$. Moreover, if there were another zero of $f$ on the line $\{\Re(z) = 1\}$, it would follow that $\{\alf_j\}_j$ is rank one, a contradiction.

The  role of  the poles of the Mellin transform is illustrated in the proof of the following result.
	\begin{prop}
		For any given collection of positive numbers $\{\alf_j\}_{j \in \cali}$ which sum to 1 such that $H = -\sum_j\alf_j\log(\alf_j) <\infty$, there is no $\eps>0$ for which
			\begin{equation}
				|A_\la| = \frac 1 {H\la} + \calo(\la^{1-\eps})
				\label{eq - the wrong exponential asymptotics}
			\end{equation}
		as $\laz$.
	\end{prop}
	\begin{proof}[Proof of Proposition 1]
		Fix a collection $\{\alf_j\}_{j \in \cali}$ as above and assume for contradiction that there is an $\eps>0$ for which \eqref{eq - the wrong exponential asymptotics} holds. Taking the Mellin transform of this equation yields that 
			$$
					\frac	1	{zf(z)}
				=
					\frac	1	{H(z-1)}
				+
					\int_1^\infty t^{-z-\eps}\calo(1)\ \id t.
			$$
		Since this second integral converges absolutely for all $z$ with $\Re(z) > 1-\eps$, we see that the left hand side has a meromorphic extension to the half plane $\{z \in \CC \;:\; \Re(z) > 1 - \eps\}$, 
		with only one pole at $z=1$. To obtain a contradiction, we provide a sequence of zeros $z_n = u_n + iv_n$ with $u_n \to 1$, $v_n \to \infty$, using the theory of almost-periodic functions. In particular, we follow the proof of the corollary to \cite[Theorem 3.6]{corduneanu}.
		
		It can quickly be seen that $f$ is almost-periodic, by \cite[Cor. to Thm. 3.12]{corduneanu}, and is bounded on $\{\Re(z) \geq 1-\eps/2\}$ since, in this half-plane,
			$$
				|f(z)| \leq 1 + \sum_{i \in \cali} \alf^{1 - \,\eps/2}.
			$$
		%
		
		Also, since $f(1) = 0$, the definition of almost-periodicity provides a sequence of positive numbers $(y_n)_{n=1}^\infty$ for which $f(1 + iy_n) \to 0$ as $n \to \infty$. 
		In other words, the holomorphic functions
			$$
				f_n(z) =  f(z + iy_n)
			$$
		are bounded on the same half-plane, and $f_n(1) \to 0$. 
		
		Furthermore, fixing any index $j \in \cali$, for $L := -2\pi/\log(\alf_j)$, we see that $\sup_{|v|\leq L}|f_n(1 + iv)|$ is bounded away from zero uniformly, since every interval of length $L$ contains a $v$ such that
			$$
					\alf_j^{1 + i v}= -\alf_j
				\quad\therefore\quad
					\left|
						f \left(
							1 + i v
						\right) 
					\right|
				\geq
					1 + \alf_j - \sum_{k \in \cali\setminus\{j\}}\alf_k = 2\alf_j.
			$$
		Therefore, on the rectangle $(1-\eps,1+\eps) + i(-l,l)$, an application of Montel's theorem shows that $f_n$ uniformly converges 
		(passing to a subsequence if necessary) to some analytic function $f_\infty$. From the last two considerations, $f_\infty (1) = 0$ and $f_\infty$ is not identically zero.
		
		Now, taking a circle about 1 small enough that $f_\infty$ has no zeros on it, by Hurwitz's theorem, for all $n$ sufficiently large, $f_n$ has a zero $\hat z_n$ inside this circle, such that $\hat z_n \to 1$ as $n \to \infty$.
		
		This thus provides us with a sequence of zeros of $f$ accumulating on the line as required, contradicting the statement that $f$ has only one zero in the above-mentioned half-plane.
	\end{proof}
Considering now the proof of Theorem 3, the following lemma uses the $(2+r)$-badly approximable hypothesis, following \cite{drmota}. For simplicity in the following two proofs, we write $\alf = \max(\alf_j,\alf_k)$ and $\bet = \min(\alf_j,\alf_k)$. 
	\begin{lem}
		Suppose that $(\alf_j)_{j \in \cali}$ satisfies the conditions of Theorem 3. Then there exists $C>0$ such that, whenever $z = 1 - u + iv \in \CC \setminus\{1\}$ satisfies both $f(z)=0$ and $u < \eps$, then $u > 0$ and
			$$
				|v|^{2+2r} \geq \frac Cu.
			$$
		\label{lem - zero free region}
	\end{lem}
\begin{proof}[Proof of Lemma \ref{lem - zero free region}]
	  The fact that $u>0$ follows from the discussion preceding the previous proposition. We first show, if $f(z) = 0$, the argument of $\alf^z$ in $(-\pi,\pi]$ is $\calo(\sqrt u)$ as $u \to 0^+$. That is, the quantity $\eta_\alf \in (-\pi,\pi]$, satisfying
		$$
				e^{i\eta_\alf}
			= 
				\frac{\alf^z}{|\alf^z|}
			= 
				\frac
					{\alf^z}
					{\alf^{1-u}},
		$$
	is $\calo(\sqrt u)$. This uses the triangle inequality and a small amount of trigonometry, as we now detail. We have that
		$$
			|\alf^z - 1| \leq 1 - \alf + H(u),
		$$
	where 
		$$
		 	H(u) := \sum_{n \in \cali}\alf_n^{1-u} - \alf^{1-u} - 1 +\alf.
		 $$
	In particular, $H(u) = \calo(u)$ as $u \to 0^+$, by the mean value theorem. Therefore, for $u$ sufficiently small, $H(u) < \alf$, which gives rise to the picture in Figure \ref{fig-complex trigonometry for eta alf}(i). Consequently, $|\eta_\alf|< \theta$, where $\theta$ is as in \ref{fig-complex trigonometry for eta alf}(ii) and satisfies the following equation.
		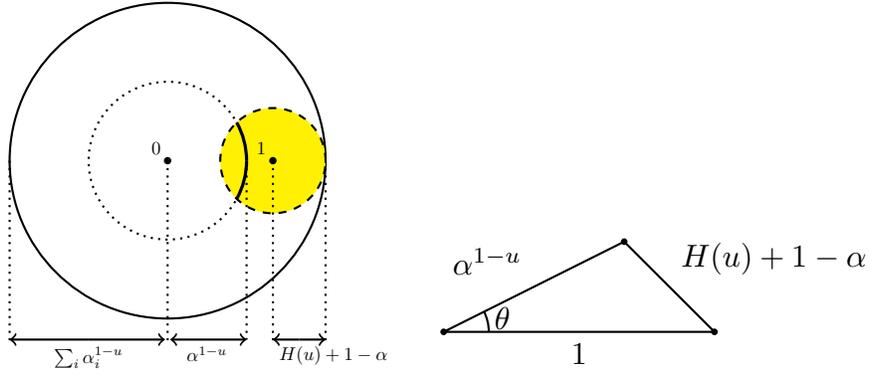
\begin{figure}
			\centering 
		\begin{tikzpicture}[thick,scale=0.7, every node/.style={scale=0.7}]
					\draw[thick] (0,0) circle (3);
					\draw[thick,dotted] (0,0) circle (1.5);
					\draw[thick,dashed,fill=yellow] (2,0) circle (1);
					\draw[dotted] (0,0) -- (0,-3.5);
					\draw[dotted] (-3,0) -- (-3,-3.5);
					\draw[thick,<-] (-3,-3.4) -- (-1.5,-3.4) node[anchor=north] {$\sum_i \alf_i^{1-u}$};
				\draw[thick,->] (-1.5,-3.4) -- (-0.05,-3.4);
						\draw[dotted] (2,0) -- (2,-3.5);
					\draw[dotted] (3,0) -- (3,-3.5);
					\draw[thick,<-] (2,-3.4) -- (2.5,-3.4);
					\draw (3.15,-3.4) node[anchor=north] {$H(u) + 1 - \alf$};
					\draw[thick,->] (2.5,-3.4) -- (3,-3.4);
					\draw[dotted] (1.5,0) -- (1.5,-3.4);
					\draw[thick,<-] (0.05,-3.4) -- (0.75,-3.4) node[anchor = north]{$\alf^{1-u}$};
					\draw[thick,->] (0.75,-3.4) -- (1.5,-3.4);
					\draw[fill] (0,0) circle [radius=0.05] node[anchor=south east] {$0$};
					\draw[fill] (2,0) circle [radius=0.05] node[anchor=south east] {$1$};
					\draw[very thick] (1.5,0) arc (0:29:1.5);
					\draw[very thick] (1.5,0) arc (0:-29:1.5);
			\end{tikzpicture} \hskip 0.5cm
				\begin{tikzpicture}[thick,scale=1.2, every node/.style={scale=1.2}]
					\draw[thick] (0,0) --
					(3,0) -- (2,1) -- (0,0);
					\draw[fill] (0,0) circle [radius=0.025] ;
					\draw[fill] (3,0) circle [radius=0.025] ;
					\draw[fill] (2,1) circle [radius=0.025] ;
					\draw (1.5,0) node[anchor=north] {$1$};
					\draw (2.5,0.5)  node[anchor=south west] {$H(u) + 1 - \alf$};
					\draw (1,0.5)  node[anchor=south east] {$\alf^{1-u}$};
					\draw (0.5,0) arc (0:27:0.5);
					\draw (0.65,0.15) node {$\theta$};
			\end{tikzpicture}
		\caption{(i) The region in which $\alf^z$ must lie for $f(z) = 0$---the bold arc in the shaded circle; (ii) the triangle defining $\theta(u)$, the maximum possible value of $|\eta_\alf|$.}
		
		\label{fig-complex trigonometry for eta alf}
		
		\end{figure}
		$$
				\cos(\theta)
			=
				\frac{1 + \alf^{2 - 2u} - (1-\alf+H(u))^2}{2\alf^{1-u}} = 1 - \calo(u)\quad (u \to 0),
		$$
	where the constant now depends on the value of $\alf$. Using, for example, that
		$$
				\lim_{y \to 0^{+}} 
					\frac
						{\arccos(1 - y)}
						{\sqrt y} 
			= 
				\sqrt 2,
			\qquad 
				\arccos:[-1,1]\to [0,\pi],
		$$
	it is clear that $\theta = \calo(\sqrt u)$ as $u \to 0^+$, hence the same applies to $\eta_\alf$.
	
	We can repeat this argument with $\bet$ in place of $\alf$ to bound the analogously defined $\eta_\bet$---i.e., $\eta_\bet = \calo(\sqrt u)$.
	
	Write $v\log(\alf) = 2\pi k + \eta_\alf$ and $v\log(\bet) = 2\pi l + \eta_\bet$, supposing $|v|\geq 2\pi/\log(\bet)$ so that $k$ and $l$ are non-zero. Substituting into the definition of $(2+r)$-badly approximable gives
		$$
			\frac d {|l|^{2+r}} 
		\le
			\left|
				\frac{\log(\alf)}{\log(\bet)} - \frac k l 
			\right|
		=
			\left|
				\frac{2\pi k + \eta_\alf}{2\pi l + \eta_\bet} - \frac k l 
			\right| 
		= 
			\left|
				\frac {\eta_\alf}{2\pi l + \eta_\bet} - \eta_\bet \frac {2\pi k +\eta_\alf}{(2\pi l + \la)^2}
			\right|,
		$$
	where $\la \in \RR$ is a constant, provided by the mean value theorem, satisfying $0<|\la|< |\eta_\bet| \le \pi$. Using the triangle inequality on the right hand side, multiplying through by $(2\pi l + \la)^2\,|l|^{r}$ and using that $|\eta_\alf| \le \pi$,  one obtains
		\begin{align*}
					\pi^2 d
				\leq 
					4\pi^2 \left(1 - \frac {\la}{l}\right)^2d
				&\leq
					2\pi 
					\left(
						|\eta_\alf|
							\frac {\left(1 - \la/2\pi l\right)^2}{|1 + \eta_\bet/2\pi l|}
					+ 
						|\eta_\bet|
							\left|\frac k l - \frac{\eta_\alf}{2\pi} \right|
					\right)
				|l|^{1+r}
			\\ &\leq
					2\pi
				\left(
					9|\eta_\alf|
					+ 
					|\eta_\bet|
					\left(\left|\frac k l\right| + \frac 1 2 \right)
				\right)
					|l|^{1+r}
			\\ 
				&\leq			
					2\pi
				\left(
					9|\eta_\alf|
				+ 
					|\eta_\bet|
				\left( \frac 3 2 + 2 \frac{\log(\alf)}{\log(\bet)}
					\right)
				\right)	
					|l|^{1+r},
		\end{align*}
where the last inequality uses that $|v| > -2\pi/\log(\bet)$.
We can divide through by the large bracket on the right hand side and, recalling the $\sqrt u$ asymptotic for $\eta_\alf$ and $\eta_\bet$, obtain the required inequality for some constant $C$, for $u$ sufficiently small and $|v|$ sufficiently large.

That the inequality holds in the whole of the specified region (with a possibly different $C$) follows simply from the fact that zeros of $f$ can only accumulate on the vertical boundary $\{z = 1-\eps + iv : v \neq 0\}$, and so there is an open neighbourhood of $[1-\eps,1]$ containing only one zero of $f$, at 1. Since there are finitely many zeros of $f$ to cater for (at most), we can adapt $C$ accordingly.
\end{proof}
The next lemma is a variant on the last and allows us to estimate decay of the Mellin inverse integral inside the zero-free region.
\begin{lem}
	Suppose $(\alf_j)_{j \in \cali}$ is a collection of positive numbers as given in Theorem 3. Then there exists $C>0$ such that, whenever $z = 1 - u + iv\in \CC$ with $u \geq 0$ and $\sig > 0$ sufficiently small,
		\begin{equation}
			\left| f(z) \right| < \sig
		\end{equation}
	implies one of the following holds: either
		$$
			|v|\leq 2\pi/\log(\bet)
		$$
	or 
		$$
			|v|^{2+2r} > \frac C{\max(u,\sig)}.
		$$
	\label{lem - growth in the zero-free region}
\end{lem}
\begin{proof}[Proof of Lemma \ref{lem - growth in the zero-free region}]
	The proof is an adaptation of that for Lemma \ref{lem - zero free region}. This time, the radius of the circle depicted in Figure \ref{fig-complex trigonometry for eta alf} is $H(u) + 1 - \alf + \sig$ and correspondingly,
		$$
			\cos(\theta) = \frac12 \frac{1 + \alf^{2 - 2u} - (1-\alf+H(u) + \sig)^2}{\alf^{1-u} } = 1 - \calo(\max(u,\sig))
		$$
	as $\max(u, \sig) \to 0$, which gives, for $\max(u, \sig)$ is sufficiently small),
		$$
			|\theta| \leq \frac \pi 2 \sqrt {1 - \cos(\theta)} = \calo(\sqrt{\max(u,\sig)})
		$$
	and the proof continues as before.
\end{proof}
The next crucial lemma is the analogue to Lemma 6 in the higher rank case.
\begin{lem}
	Under the assumptions of Theorem 3,
		$$
			|A_\la| = \frac 1 {H \la} + \calo\big(\la^{-1}(-\log(\la))^{-P}\big)
			\qquad
			\la \in (0,1],
		$$
	where $P \in (0,P^*)$ is as given in Theorem 3.
\end{lem}
\begin{proof}[Proof of Lemma 10]
	The proof uses simple complex analysis to estimate the integral
		\begin{equation}
			F(t) := \frac 1 {2\pi i}\int_{2 - i\infty}^{2 + i\infty} \frac{t^{z+3}}{z(z+1)(z+2)(z+3)f(z)} \id z,
			\hbox{ for }  t > 1.
		\label{eq - Mellin inverse for F}
		\end{equation}
	We first relate $F(t)$ to $|A_{1/t}|$. On the line $\{\Re(z) = 2\}$, we see that $f$ is uniformly bounded away from zero,
		$$
			|f(z)| \geq 1 - \sum_{j \in \cali}\alf_j^2 > 0,
		$$
	so $F(t)$ absolutely converges, for all $t$. Therefore, by the Mellin inversion theorem, the Mellin transform of  $t \mapsto F(t)/t^3$ is the denominator of the integrand of $F$:
		$$
			F^\ast(z) := \int_0^\infty t^{- z - 1}F(t)\ \id t= \frac 1{z(z+1)(z+2)(z+3)f(z)}.
		$$
	Therefore, using integration by parts, 
	one has $F^{(3)}(t) = |A_{1/t}|$ Lebesgue almost-everywhere. 
	
	We now relate the integral in \eqref{eq - Mellin inverse for F} to that over the contour $\Ga$, parametrised by 
		$$
			\gamma:\RR \to \CC, \qquad \gamma(v) = 1 + iv - D\min(1, |v|^{-2-r})
		$$
	(see Figure \ref{fig - Gamma contour}), where $D>0$ is chosen sufficiently small so that the previous two lemmas apply as follows: firstly, the only zero of $f$ which lies on or to the right of $\Ga$ is at $1$, and secondly, whenever $z$ lies on or to the right of $\Ga$ and $|\Im(z)| \geq -2\pi(\log(\bet))^{-1}$, one has
		\begin{equation}
			|f(z)| \geq D |\Im(z)|^{-2-2r}.
				\label{eq - vertical decay of f z}
		\end{equation}
	Consider, for $T\geq 1$, the contour $\Gamma_T$ depicted in Figure \ref{fig - Gamma contour}.
	%
		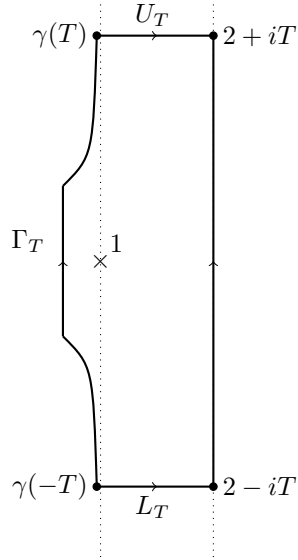
\begin{figure}
		$$
			\begin{tikzpicture}
				
					\draw[dotted,thin] (0,3.5) -- (0,-4);
					\draw[dotted,thin] (1.5,3.5) -- (1.5,-4);
				
					\draw[thick] (-0.05,-3) -- (1.5,-3) -- (1.5,3) -- (-0.05,3);
					\draw[thick] (-0.5,-1) -- (-0.5,1);
					\draw[thick] (-0.5,1) .. controls (-0.1,1.4) .. (-0.05,3);
					\draw[thick] (-0.5,-1) .. controls (-0.1,-1.4) .. (-0.05,-3);
				
					\draw[fill] (-0.05,-3) circle [radius=0.05]
						node[anchor = east] {$\ga(-T)$};
					\draw[fill] (-0.05,3) circle [radius=0.05]
						node[anchor = east] {$\ga(T)$};
					\draw[fill] (1.5,3) circle [radius=0.05]
						node[anchor = west] {$2 + iT$};
					\draw[fill] (1.5,-3) circle [radius=0.05]
						node[anchor = west] {$2 - iT$};
					
				\draw (0,0) 
					node {$\times$} 
					node[anchor = south west] {$1$};

				\draw[->] (-0.5,-0.01) -- (-0.5,0)
					node[anchor = south east] {$\Ga_T\ $};
				\draw[->] (0.74,3) -- (0.75,3)
					node[anchor = south] {$U_T\ $};
				\draw[->] (0.74,-3) -- (0.75,-3)
					node[anchor = north] {$L_T\ $};
				\draw[->] (1.5,-0.01) -- (1.5,0);
				
			\end{tikzpicture}
		$$
		\caption{The contour $\Ga_t$ used in the proof of Lemma 10.}
		\label{fig - Gamma contour}
	\end{figure}
	Cauchy's residue theorem gives 
		$$
				\int_{2-iT}^{2+iT} F^\ast(z) \;\id z
			= 
				{2\pi i}\frac {t^4} {24H} + 
				\int_{\Ga_T} F^\ast(z) \,\id z +
				\int_{U_T} F^\ast(z) \,\id z + 
				\int_{L_T} F^\ast(z) \,\id z.
		$$
	Since $U_T$ and $L_T$ have bounded length, a simple application of $\eqref{eq - vertical decay of f z}$ shows that the corresponding last two integrals are $\calo (|T|^{-2+2r})$ as $T \to \infty$; therefore, taking the limit, one has
		\begin{equation}
			F(t) = 
				\frac {t^4} {24H} 
				+ 
				\frac 1 {2\pi i}\int_{\Ga} F^\ast(z)t^{-z-3} \;\id z
			\label{eq - equation for F}.
		\end{equation}
	From this point, the proof follows along the lines of \cite[Theorem 4.6, pp.133--4]{primes}. Since $t \mapsto |A_{1/t}|$ is non-decreasing on the real line, so too are the functions $F$, $F'$ and $F''$. Using this, together with repeated applications of the mean value theorem, gives that, for any $t\in \RR$ and $h>0$, both 
		\begin{align*}
				\frac{F(t-3h) - 3F(t-2h) + 3F(t-h) - F(t)}{-h^3}
			\leq
				|A_{1/t}|& \quad\text{and}\\
				|A_{1/t}|
			\leq
				\frac{F(t+3h) - 3F(t+2h) + 3F(t+h) - F(t)}{h^3}&\quad \text{hold.}
		\end{align*}
	Substituting \eqref{eq - equation for F} into these expressions yields
		\begin{align*}
				\frac{F(t\pm 3h) - 3F(t\pm 2h) + 3F(t \pm h) - F(t)}{\pm h^3} 
			= \qquad\qquad\qquad & \\
			 	t \pm \frac 3 2 h 
			+
				\frac 1 {2 \pi i}
				\int_\Ga
					\frac
						{(t\pm 3h)^{z+3} - 3(t\pm 2h)^{z+3} + 3(t \pm h)^{z+3} - t^{z+3}}
						{\pm h^3}&
					F^\ast(z)
				\id z.
		\end{align*}
	Thus, we have the estimate
		\begin{align*}
				|A_{1/t}|
			& = 
				\; t + \calo(h) \, +
			\\	
			\;
				\calo & \left(
					\int_\Ga
						\left|
							\frac
								{(t\pm 3h)^{z+3} - 3(t\pm 2h)^{z+3} + 3(t \pm h)^{z+3} - t^{z+3}}
								{\pm h^3 (z+1)(z+2)(z+3)}
						\right|	
						\left|
							\frac
								1
								{zf(z)}
						\right|\;
					|\id z|
				\right).
		\end{align*}
	From now on, let $h = h(t) \in (0,t)$ be a function of $t$ to be determined later. To begin to estimate the integral, consider $|\De_\pm(t,h,z)|$ for $t > 1$ and $z \in \Ga$, where
		$$
				\De_\pm(t,h,z) 
			:= 
				\frac
					{(t\pm 3h)^z - 3(t\pm 2h)^z + 3(t \pm h)^z - t^z}
					{\pm h^3 (z+1)(z+2)(z+3)}.
		$$
	We estimate $|\De_\pm|$ in two different ways. For the first, we express $|\De_\pm|$ as a series of nested integrals:
		\begin{align*}
				\left|
					\De_\pm(t,h,z)
				\right|	
			= 
				\left|
					\frac 
						1 {\pm h}
					\int_t^{t \pm h}
						\frac
							1 {\pm h}
						\int_{\hat t}^{\hat t \pm h}
							\frac
								1 {\pm h}
							\int_{\subdoublehat}^{\subdoublehat \pm h}
								\triplehat^z  \ 
							\id \triplehat \
							\id \doublehat \
							\id \hat t
				\right|
			& \leq
				(t + 3h)^{\Re(z)}
			\\	
			& \leq
				4 t^{\Re(z)}.
		\end{align*}
	For the first inequality, we have taken the absolute value signs inside the integral and applied the mean value theorem three times, noting that $t \mapsto |t^z| = t^{\Re(z)}$ is increasing on $[0,\infty)$, since $\Re(z) > 0$ for $z \in \Ga$. (The second follows simply from $h < t$ and $\Re(z) < 1$.)
	
	The triangle inequality gives us another estimate: for $t \geq 1$,
		\begin{align*}
				\left|
					\De_\pm(t,h,z)				
				\right|	
			& \leq
				\frac
					{(4t)^{Re(z)+3} + 3(3t)^{Re(z)+3} + 3(2t)^{Re(z)}  + t^{Re(z)+3}}
					{h^3 |z+1||z+2||z+3|}
			\\
			& \leq 
				\frac
					{548 t^{\Re (z)+3}h^{-3}}
						{|z+1||z+2||z+3|}
				.
		\end{align*}

	From the Bernoulli inequality it follows that, for $z = \ga(v) \in \Ga$, the three quantities $|\ga(v)+1|,|\ga(v)+2|,|\ga(v)+3|$ are all greater than or equal to $(1 + |v|)/2$, so altogether we have the following.
		$$
				|\De_\pm\big(t,h,\ga(v)\big)|		
			\leq 
				\min 
				\left(
					4 t^{\Re(\ga(v))},
					\frac
						{4384 \, t^{\Re (\ga(v))+3}}
						{h^3 \, (1 + |v|)^3}
				\right).
		$$
	From \eqref{eq - vertical decay of f z}, one can easily deduce that  $\big|\ga(v)f\big(\ga(v)\big)\big|^{-1} = \calo ((1 + |v|)^{1+2r})$, for all $v\in \RR$. Combining the previous three inequalities and using the boundedness of $|\ga'(v)|$ gives the following.
		$$	
				|A_{1/t}|
			=
				\frac tH 
			+ 
				\calo(h)
			+ 
				\calo
				\left(	
					\int_{-\infty}^\infty 
						(1 + |v|)^{1 + 2r}
						\min \left(
							t^{\Re (\ga(v))},
							\frac
								{t^{\Re (\ga(v))+3}}
								{(1+|v|)^3}
						\right)
					\id v 
				\right).
		$$
	Because the integral is symmetric in $v$, it suffices to estimate the integral from 0 to $\infty$, as we now do. Writing $\Re(\ga(v)) = 1 - \ka(v)$, where $\ka(v) = D\min(1,|v|^{-2 - 2r})$, the previous equation simplifies to the following.
		\begin{align*}
				\frac
					{|A_{1/t}|}
					{t}
			-
				\frac1H
			& =
				\calo
					\left(
						\frac ht
					\right)
			+ 
				\calo
				\left(	
					\int_0^\infty 
						(1 + v)^{1 + 2r} t^{-\ka(v)}
						\min \left(
							1,
							\frac
								{(t/h)^3}
								{(1 + v)^3}
						\right)
					\id v 
				\right)
			\\
			& = 
				\calo
				\left(
					\frac ht
				\right)
			+ 
				\calo
				\left(	
					\int_0^\infty 
						(1 + v)^{-2 + 2r} t^{-\ka(v)}
						\min \left(
						(1 + v)^3,
						\left(
							\frac t h
						\right)^{\!\!3}
						\right)
					\id v 
				\right)
			\\
			& =
				\calo
				\left(
					\frac ht
				\right)
			+ 
				\calo
				\left(	
					\int_1^\infty 
						v^{-2 + 2r} t^{-\ka(v-1)}
						\min \left(
							v^3,
							\left(
								\frac t h
							\right)^{\!\!3}
						\right)
					\id v 
				\right).
		\end{align*}
	Now let $\de \in (0,1 - 2r)$. For $v, t \geq 1$, we have, since $\ka$ is decreasing on $[0,\infty)$,
		\begin{align*}
				v^{-2 + 2r} t^{-\ka(v-1)} 
			& \leq 
				v^{-2 + 2r} t^{-\ka(v)} 
			\\
			& = 
				v^{-2 + 2r} \exp (-\ka(v)\log(t))
			\\
			& =
				v^{-2 + 2r + \de} \exp (-\ka(v)\log(t) - \de \log(v))
			\\
			& = 
				v^{-2 + 2r + \de}
				\exp \left(
						- Dv^{-2-2r}\log(t) - \de \log(v))
				\right)	
			\\
			& = 
				v^{-2 + 2r + \de} \exp(-\xi_\de(t)),
		\end{align*}
	where
		\begin{align*}
				\xi_\de(t)
			& := 
				\inf_{v \geq 1}
				\left(
					Dv^{-2-2r}\log(t) + \de \log(v)
				\right)
			\\	
			& = 
				\frac
					\de
					{2 + 2r} 
				\left(
					1
				+
					\log
					\left(
						\frac 
							{D(2 + 2r)\log(t)}
							\de
					\right)
				\right).
		\end{align*}
	This last equality holds for all $t$ sufficiently large, by elementary calculus. Hence
		$$
				\frac
					{|A_{1/t}|}
					{t}
			-
				\frac1H
			=
				\calo
				\left(
					\frac ht
				\right)
			+ 
				\calo
				\left(
					e^{-\xi_\de(t)}	
					\int_1^\infty 
						v^{-2 + 2r + \de}
						\min \left(
							v^3,
							\left(
								\frac t h
							\right)^3
						\right)
					\id v 
				\right).
		$$
	Now, writing $\om = vh/t$ and substituting, the integral becomes
		$$
				\left(
					\frac th
				\right)^{2 + 2r + \de}
				\int_{h/t}^\infty 
					\om^{-2 + 2r + \de} \min(\om^{3},1)\; \id \om,
		$$
	and this can be split into two parts,
		$$
				 	\int_{h/t}^1 
				 		\om^{1 + 2r + \de}\; 
				 	\id \om 
				 + 
				 	\int_1^\infty 
				 		\om^{-2 + 2r + \de} \;
				 	\id \om,
		$$
	both of which are finite, since $2r + \de \in(0,1)$. Therefore, for all $t \geq 1$,
		$$
				\frac
					{|A_{1/t}|}
					{t}
			-
				\frac1H
			=
				\calo
				\left(
					\frac ht
				\right)
			+
				\calo
					\left(
						e^{-\xi_\de(t)}
						\left(
							\frac th
						\right)^{2 + 2r + \de}
					\right).
		$$
	Finally, choosing $h(t) = t \exp ( \frac{-\xi_\de(t)}{3 + 2r + \de} )$ ensures both terms have the same order of magnitude, and the previous equation simplifies to the required expression:
		$$
				|A_{1/t}|
			=
				\frac tH
			+
				\calo \left(
					t\log(t)^{-P}
				\right),
		$$
	where
		$$
				P = P(r, \de) 
			= 
				\frac 
					\de
					{(2 + 2r)(3 + 2r + \de)} \in \left(0,\frac {1-2r}{8(1+r)}\right). \qquad \qedhere
		$$
\end{proof}
	The concluding stages of the proof of Theorem 3 are similar to those of Theorem 2; but with some notable differences.
\begin{proof}[Proof of Theorem 3]
 	For simplicity of writing, we again consider only the case that none of the $T_i$ fix 0.
	Let $\la \in (0,1)$ and recall $I = [b,1)$, $V_k$ and $V$ from the proof of Theorem 2 (p.\pageref{eq - power sum with n_i's} onwards).
	
	Similarly to that proof, we may write the discrepancy in terms of three sums,
			\begin{align}
					\mu_\la (I) - \|I\|
				& =  
					\sum
						_{\substack{
							\boldv \in V \\
						\hphantom{\la \leq \alf_{\boldv} < \la/\alf_{\max}}		
						}}
						\mu_\la(T_\boldv[0,1)) - \alf_\boldv
					\nonumber
				\\
				& =
					\sum
					_{\substack{
						\boldv \in V \\
						\alf_\boldv \geq \la/\alf_{\max}
								     \\
						\hphantom{\la \leq \alf_{\boldv} < \la/\alf_{\max}}
					}}
						\frac
							{|X_{\la/\alf_\boldv}|}
							{|X_{\la}|}
					- 
						\alf_\boldv
				\nonumber
				\\
				& +
					\sum
					_{\substack{
						\boldv \in V \\ 
						\la \leq \alf_{\boldv} < \la/\alf_{\max}
					}}
						\frac
							{|X_{\la/\alf_\boldv}|}
							{|X_\la|}
				\quad
				-
					\sum
					_{\substack{\boldv \in V \\ \alf_{\boldv} < \la/\alf_{\max}}}
						\alf_\boldv,
					\label{eq-three sums for discrepancy}
			\end{align}
	where $\alf_{\max} = \max_{i\in \cali} (\alf_i)$.
	
	The second and third sum of \eqref{eq-three sums for discrepancy} decay much faster than the first, which can be deduced from the summability of $\alf_\boldv^{1-\eps}$ alone. For the third sum, the argument from Theorem 2 (p.\pageref{page-uncounted intervals decay}) applies to bound it by a multiple of $\la^{1-\eps}$. The second sum of \eqref{eq-three sums for discrepancy} corresponds to newly split intervals: if $\alf_\boldv < \la/\alf_{\max}$, then $\alf_\boldv\alf_i < \la$ for any $i \in \cali$, so no proper subinterval of $T_\boldv[0,1)$ has been split for this value of $\la$, and $|X_{\la/\alf_\boldv}| = 1$. Therefore, assuming $b$ has an infinite itinerary,
		\begin{align*}
				\sum_{\substack{v \in V \\ \la \leq \alf_\boldv < \la/\alf_{\max}}}
						{|X_{\la/\alf_\boldv}|}
			&= 
				\phantom{\sum_{n \in \NN}}
				\big|\{\boldv \in V : \la \leq \alf_\boldv < \la/\alf_{\max}\}\big|
			\\
			&=
				\sum_{n \in \NN}
					\big|\{\boldv \in V_n : \la \leq \alf_\boldv < \la/\alf_{\max}\}\big|
			\\
			&\leq
				\sum_{n \in \NN}
					\big|\{i \in \cali : \la \leq \alf_{\It_{n-1}(b)}\alf_i <  \la/\alf_{\max}\}\big|
			\\
			&=
				\sum_{n \in \NN}
					\big|\{i \in \cali : \la/\alf_{\It_{n-1}(b)}\leq \alf_i <  \la/(\alf_{\max}\alf_{\It_{n-1}(b)})\}\big|
			\\
			&\leq
				\sum_{n \in \NN}
					\big|\{i \in \cali : \la/\alf_{\It_{n-1}(b)}\leq \alf_i <  \la/\alf_{\It_{n}(b)}\}\big|
			\\
			&= 
				\phantom{\sum_{n \in \NN}}
					\big|\{i \in \cali : \la \leq \alf_i\}\big|
			\\		
			&\leq \sum_{i \in \cali} \alf_i^{1-\eps} \; \la^{\eps-1} .
		\end{align*}
	Thus, the second sum of \eqref{eq-three sums for discrepancy} is bounded by a multiple of $\la^{\eps}$. The finite itinerary case is similar, involving a finite sum in $n$.
	
	It remains to bound the first sum of \eqref{eq-three sums for discrepancy}, using the asymptotics for $|A_{\la/\alf_\boldv}|$ provided by the previous lemma. One finds that there exists $C,C'$ such that, for all $\boldv \in V$ with  $\alf_\boldv \geq \la/\alf_{\max}$,
		$$
				\left|
					\frac
						{|X_{\la/\alf_\boldv}|}
						{|X_{\la}| }
				- 
					\alf_\boldv
				\right|				
			\leq
					C\alf_\boldv
				\left(
					(-\log(\la/\alf_\boldv))^{-P}
				\right)
			\leq
					C'\alf_\boldv
				\left(
					(-\log(\la))^{-P}
				\right),					
		$$
	the last inequality following from the fact that $x \mapsto \log(\la)/\log(\la/x)$ is uniformly bounded on $[\alf_{\max},\infty)$. Summing over $\boldv$ gives the required estimate.
\end{proof}
\end{document}